\documentclass[a4paper,12pt]{amsart}

\usepackage{graphicx,import}
\usepackage[english]{babel}

\usepackage{url}
\usepackage{hyperref}

\usepackage{todonotes}
\usepackage[T1]{fontenc}
\usepackage[utf8]{inputenc}
\usepackage{csquotes}
\usepackage{verbatim}
\usepackage{todonotes}
\usepackage{amsmath}
\usepackage{amsthm}
\usepackage{amsfonts, amssymb, fancybox, multicol, array, tabularx, amscd}
\usepackage{enumerate}
\usepackage{upgreek}
\usepackage[shortlabels]{enumitem}
\usepackage{psfrag}
\usepackage{tikz}
\newtheorem{theorem}{Theorem}[section]
\newtheorem{lemma}[theorem]{Lemma}
\newtheorem{proposition}[theorem]{Proposition}

\theoremstyle{definition}
\newtheorem{definition}[theorem]{Definition}
\newtheorem{remark}[theorem]{Remark}

\theoremstyle{remark}

\newcommand{\R}{\mathbb{R}}
\newcommand{\N}{\mathbb{N}}
\newcommand{\Z}{\mathbb{Z}}
\newcommand{\C}{\mathbb{C}}

\newcommand{\e}{\varepsilon}
\newcommand{\f}{\varphi}

\DeclareRobustCommand{\rchi}{{\mathpalette\irchi\relax}}
\newcommand{\irchi}[2]{\raisebox{\depth}{$#1\chi$}}

\newcommand{\zak}{%
  \mathbin{\vrule height 1.6ex depth 0pt width
0.13ex\vrule height 0.13ex depth 0pt width 1.3ex}
}    

\newcommand{\va}[1]{\left| #1 \right|}  
\newcommand{\nor}[1]{\left\| #1 \right\|} 

\newcommand{\sk}[1]{{#1}^{\rm{skew}}}  
\newcommand{\sym}[1]{#1^{\rm{sym}}}

\DeclareMathOperator{\dist}{dist}
\DeclareMathOperator{\Per}{Per}

\DeclareMathOperator{\Span}{Span}

\DeclareMathOperator{\Divv}{Div}
\DeclareMathOperator{\curl}{curl}
\DeclareMathOperator{\Curl}{Curl}
\DeclareMathOperator{\supp}{supp}
\DeclareMathOperator{\spt}{supp}

\newcommand{\distr}{\mathcal{D}'}
\newcommand{\matricidue}{\mathbb{M}^{2 \times 2}}

\newcommand{\lduesimm}{ \mathnormal{L}^2 (\Om;\mathbb{M}^{2 \times 2}_{\text{sym}}) }
\newcommand{\ldueanti}{ \mathnormal{L}^2 (\Om;\mathbb{M}^{2 \times 2}_{\text{skew}}) }

\newcommand{\misure}{ \mathcal{M} (\Om;\R^2) }
\newcommand{\duale}{ \mathnormal{H}^{-1} (\Om;\R^2) }
\newcommand{\strain}{\mathnormal{L}^2 (\Omega;\matricidue)}

\newcommand{\weak}{\rightharpoonup}
\newcommand{\weakstar}{\stackrel{*}{\rightharpoonup}}

\newcommand{\AS}{\mathcal{AS}}
\newcommand{\AD}{\mathcal{AD}}

\newcommand{\Om}{\Omega}

\newcommand{\cir}{\int_{\partial B_\e(x_i)}}
\newcommand{\de}{\partial}

 \title[Linearised polycrystals]{Derivation of Linearised Polycrystals from a 2D system of edge dislocations}
 
 \author[S. Fanzon]
 {Silvio Fanzon}
 \address[Silvio Fanzon]{University of Sussex, Department of Mathematics, Pevensey 2 Building, Falmer Campus,
Brighton BN1 9QH, United Kingdom}
 \email{S.Fanzon@sussex.ac.uk}

\author[M. Palombaro]
 {Mariapia Palombaro}
 \address[Mariapia Palombaro]{University of Sussex, Department of Mathematics, Pevensey 2 Building, Falmer Campus,
Brighton BN1 9QH, United Kingdom}
 \email{M.Palombaro@sussex.ac.uk}

\author[M. Ponsiglione]
 {Marcello Ponsiglione}
 \address[Marcello Ponsiglione]{Dipartimento di Matematica, Sapienza Universit\`a di Roma, 00185 Rome, Italy}
 \email{ponsigli@mat.uniroma1.it}

\begin{document}

\begin{abstract}
\small{
In this paper we show the emergence of  polycrystalline structures as a result of elastic energy minimisation. For  this purpose, we introduce a variational model for two-dimensional systems of edge dislocations, within the so-called core radius approach, and we derive the $\Gamma$-limit of the elastic energy functional as the lattice space tends to zero. 

In the energy regime under investigation, the symmetric and skew part of the strain become decoupled in the limit, the dislocation measure being the curl of the skew part of the strain.  The limit energy is given by the sum of a plastic term, acting on the dislocation density, and an elastic term, which depends on the symmetric strains. Minimisers under suitable boundary conditions are piece-wise constant antisymmetric strain fields, representing in our model a polycrystal whose grains are mutually rotated by infinitesimal angles.  

\vskip .3truecm \noindent Keywords: Geometric rigidity, Linearization, Polycrystals, Dislocations, Variational methods.
\vskip.1truecm \noindent 2000 Mathematics Subject Classification:  74B15,  74N05, 74N15, 49J45.
 

}
\end{abstract}
 
\maketitle
\tableofcontents

\section{Introduction}


Many solids in nature exhibit a polycrystalline structure. 
A \textit{single phase polycrystal} is formed by many individual crystal grains,
having the same underlying periodic atomic structure, but 
 rotated with respect to each other. 
The region that separates two grains with different orientation is called \textit{grain boundary}. 
Since the grains are mutually rotated, the periodic crystalline structure is disrupted at grain boundaries. 
As a consequence, grain boundaries are regions where {\it dislocations} occur, inducing  high energy concentration.

Polycrystalline structures, which a priori may seem energetically not convenient, arise from the crystallisation of a melt. As the temperature decreases, crystallisation starts from a number of points within the melt. These single grains grow until they meet. Since their orientation is generally different, the grains are not able to arrange in a single crystal and grain boundaries appear as local minimizers of the energy, in fact  as metastable configurations. 
After crystallisation there is a grain growth phase, when the solid tries to minimise the energy by reducing the boundary area. This process happens by atomic diffusion within the material, and it is thermally activated (see \cite[Ch 5.7]{gottstein},  \cite{Taylor}). 
On a mesoscopic scale a polycrystal resembles the structure in Figure \ref{image:poly}. 
\begin{figure}[t]
\begin{center}
\includegraphics[scale=0.35]{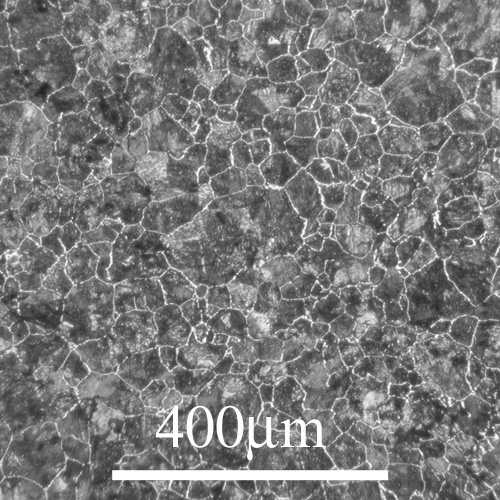}
\caption{Section of an iron-carbon alloy. The darker regions are single crystal grains separated by grain boundaries represented by lighter lines (source \cite{imagepoli}, licensed under CC BY-NC-SA 2.0 UK).} 
\label{image:poly}   
\end{center}
\end{figure}

Our purpose is to describe, and to some extent to predict, polycrystalline structures by variational principles. To this end, we first derive by $\Gamma$-convergence, as the lattice spacing tends to zero,  a total energy functional depending  on the strain and on the dislocation density. Then, we focus on the ground states of such energy,
neglecting the fundamental mechanisms driving the formation and evolution of grain boundaries. 
The main feature of the model proposed in this paper is that grain boundaries and the corresponding grain orientations are not introduced as internal variables of the energy; in fact, they  spontaneously arise as a result of the only energy minimisation under suitable boundary conditions.

Let us introduce our model by first discussing  the case of two dimensional small angle tilt grain boundaries (from now on abbreviated to SATGB). 
The atomic structure of SATGBs is well understood (see, for example,  \cite[Ch 3.4]{gottstein}, \cite{shockley}). In fact, the lattice mismatch between two grains mutually tilted by a small angle $\theta$ is accommodated by a single array of edge dislocations at the grain boundary, evenly spaced at distance $\delta \approx \e /\theta$, where $\e$ represents the atomic lattice spacing. 
Therefore, the number of dislocations at a SATGB is of the order 
$\theta /\e $
(see Figure \ref{fig:tilt}). The elastic energy of a SATGB is given by the celebrated Read-Shockley formula introduced in \cite{shockley} 
\begin{equation} \label{intro3 rs}
\rm{Elastic \,\, Energy} =  E_0 \theta  (A + |\log \theta|) \,, 
\end{equation}
where $E_0$ and $A$ are positive constants depending only on the material.
Recently Lauteri and Luckhaus in \cite{lauteri},
starting from a nonlinear elastic energy,
 proved compactness properties and energy bounds in agreement  with the Read-Shockley formula. 
\begin{figure}[t!] 
\centering   
\def\svgwidth{14.5cm}   
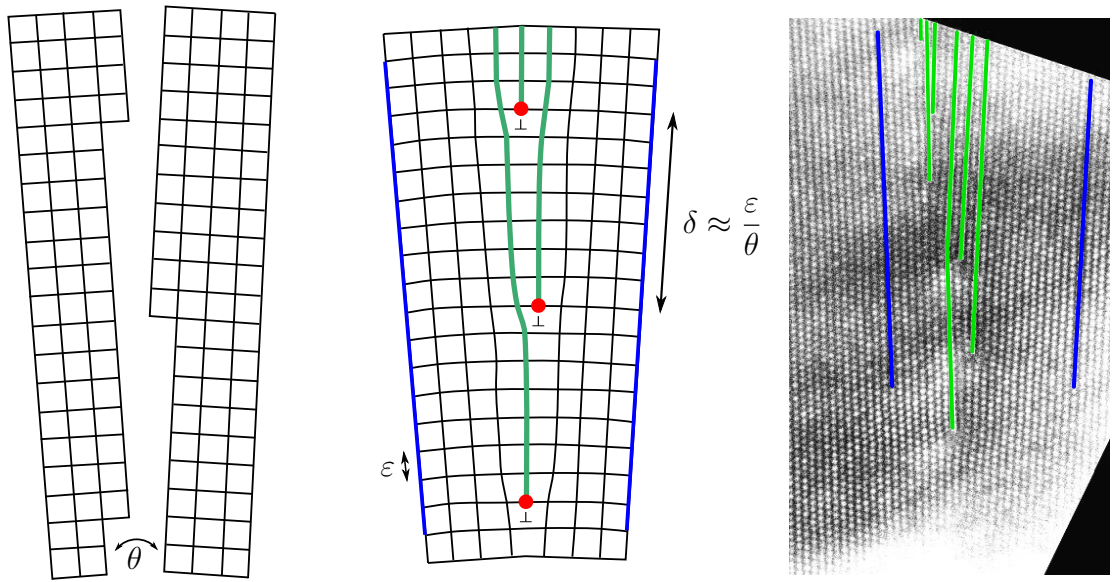  
\caption{Left: schematic picture of two grains mutually rotated by an angle $\theta$. Centre: schematic picture of a SATGB. The two grains are joined together and the lattice misfit is accommodated by an array of edge dislocations spaced  $\delta$ apart and represented by red dots (pictures after \cite{read}). Right: HRTEM of a SATGB in silicon. The green lines represent rows of atoms ending within the crystal. Their end points inside the crystal are edge dislocations, which correspond to the red atoms in the central picture. The blue lines show the mutual rotation between the grains (image from \cite[Section 7.2.2]{imagegrain} with permission of the author H. Foell).}   
\label{fig:tilt}
\end{figure}


In thiinputs paper we focus on lower energy regimes, deriving by $\Gamma$-convergence, as the lattice spacing $\e \to 0$ and the number of dislocations $N_\e \to \infty$, a certain limit energy functional $\mathcal{F}$ that can be regarded as a linearised version of 
the Read-Shockley formula.
We work in the setting of linearised planar elasticity as introduced in \cite{glp} and in particular we require good separation of the dislocation cores. Such good separation hypothesis will in turn imply that the number of dislocations at grain boundaries is of the order 
\begin{equation} \label{intro3 less disl}
N_\e \ll \frac{\theta}{\e} \,.
\end{equation}
As a consequence, we cannot allow a number of dislocations sufficient to accommodate small rotations $\theta$ between grains, but rather we can have rotations by an infinitesimal angle $\theta \approx 0$, that is, antisymmetric matrices.
In this respect our analysis represents the linearised counterpart of the Read-Shockley formula: grains are micro-rotated by infinitesimal angles and the corresponding ground states can be seen as linearised polycrystals, whose energy is linear with respect  to the number of dislocations at grain boundaries.

We now briefly introduce the setting of our problem following \cite{glp}. In \textit{linearised planar elasticity}, the reference configuration is a bounded domain $\Omega~\subset~\R^2$, representing a horizontal section of an infinite cylindrical crystal $\Om \times \R$. A displacement is a regular map $u \colon \Om \to \R^2$ and the
stored energy density $W \colon \matricidue \to [0,+\infty)$ is defined by
\[
W(F):= \frac{1}{2} \C F : F \,,
\]
where $\C$ is a fourth order stress tensor that satisfies 
\[
c^{-1} |\sym{F}|^2 \leq \C F : F \leq c |\sym{F}|^2  \quad \text{for every} \quad F \in \matricidue \,.
\]
Here $\sym{F}:=(F+F^T)/2$ and $c$ is some positive constant. The energy density $W$ acts on gradient strain fields $\beta := \nabla u$ and the elastic energy induced by $\beta$ is defined as 
\[
\int_{\Om} W(\beta) \, dx \,. 
\]

Following the \textit{semi-discrete dislocation model} (see \cite{cermelli,deshpande,glp}), dislocations are introduced
 as point defects of the strain $\beta$. 
 More specifically, 
a straight dislocation line $\gamma$ orthogonal to the cross section $\Om$ is identified with the point $x_0 = \gamma \cap \Om$. We then require
\begin{equation} \label{curl non zero}
\Curl \beta = \xi \, \delta_{x_0} \,,	
\end{equation}
in the sense of distributions. Here $\xi:=(\xi_1,\xi_2,0)$ is the Burgers 
vector, orthogonal to $\gamma$, so that $(\gamma,\xi)$ defines an edge dislocation. Therefore, with the identification above, also $(x_0,\xi)$ represents an edge dislocation (see Figure \ref{fig inf cil}). 
\begin{figure}[t!] 
\centering   
\def\svgwidth{13.5cm}   
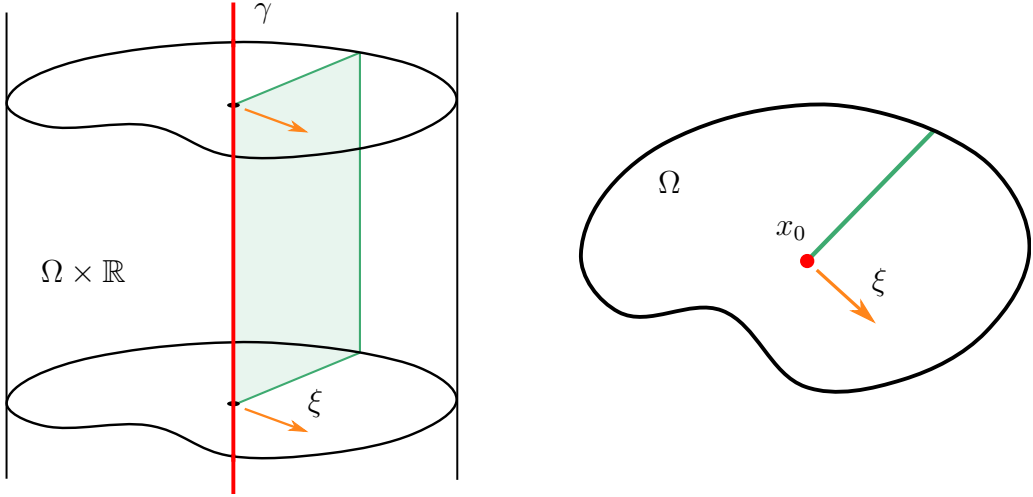  
\caption{Left: cylindrical domain $\Om \times \R$. The dislocation $(\gamma,\xi)$ is of edge type. The green plane represents the extra half-plane of atoms corresponding to $\gamma$. Right: section $\Om$ of the cylindrical domain in the left picture. The red point $x_0 = \gamma \cap \Om$ represents the section of the dislocation line, so that $(x_0,\xi)$ identifies an edge dislocation. The green line is the intersection of the extra half-plane of atoms in the left picture with $\Om$.}   
\label{fig inf cil}
\end{figure}
It is immediate to check that \eqref{curl non zero} implies
\[
\int_{B_\sigma (x_0) \setminus B_\e (x_0)} W(\beta) \, dx \geq c \log \frac{\sigma}{\e}  \,,  \quad \text{for every} \quad \sigma>\e >0 \,.
\] 
From the above inequality we deduce that, as $\e \to 0$, the energy diverges logarithmically in neighbourhoods of $x_0$. To overcome this problem we adopt the so-called core radius approach. Namely, we remove from $\Om$ the ball $B_\e (x_0)$, called the \textit{core region}, where $\e$ is proportional to the underlying lattice spacing, and we replace \eqref{curl non zero} by the circulation condition
\[
\int_{\de B_\e (x_0)} \beta \cdot t \, ds = \xi \,.
\]
In the above formula $t$ is the unit tangent vector to $\de B_\e (x_0)$ and $ds$ in the $1$~-~dimensional Hausdorff measure. A generic distribution of $N$ dislocations will therefore be identified with the points $\{x_i\}_{i=1}^N$. To each $x_i$ we associate a corresponding Burgers vector $\xi_i$, belonging to a finite set $\mathcal{S} \subset \R^2$ of admissible Burgers vectors, which depends on the underlying crystalline structure.  Clearly the Burgers vector scales like $\e$; for example for a square lattice we have $\mathcal{S}= \e \{ \pm e_1, \pm e_2\}$.  From now on we will always renormalise the Burgers vectors, scaling them by $\e^{-1}$, so that 
$\mathcal{S}$ becomes a fixed set independent of the lattice spacing.   
The energy is in turn scaled by $\e^{-2}$, since it is quadratic with respect to the Burgers vector. 
Following \cite{glp}, we make a technical hypothesis of good separation for the dislocation cores, by introducing a small scale $\rho_\e \gg \e$, called \textit{hard core radius}. Any cluster of dislocations contained in a ball $B_{\rho_\e}(x_0) \subset \Om$ will be identified with a multiple dislocation $\xi \, \delta_{x_0}$, where $\xi$ is the sum of the Burgers vectors corresponding to the dislocations in the cluster (see Figure \ref{fig core radius} Left). Therefore $\xi \in \mathbb{S}$, where 
\[
\mathbb{S}:=\Span_\Z \mathcal{S}
\] 
is the set of multiple Burgers vectors. 
\begin{figure}[t!] 
\centering   
\def\svgwidth{14cm}   
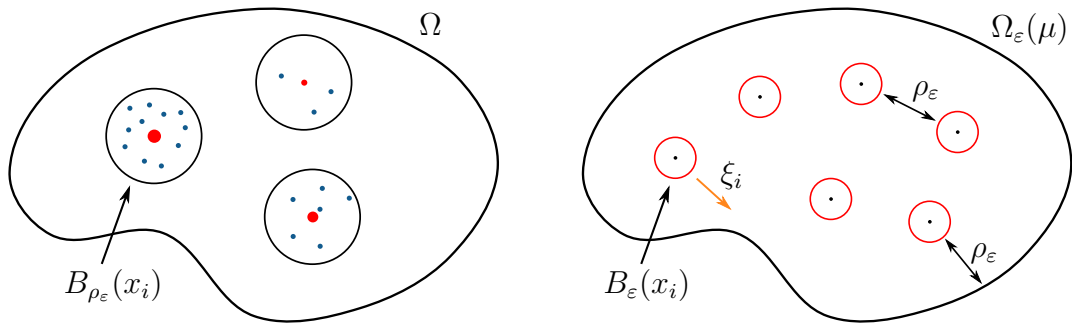  
\caption{Left: clusters of dislocations (blue points) inside the balls $B_{\rho_{\e}}(x_i)$ are identified with a single dislocation $\xi_i \, \delta_{x_i}$ centred at $x_i$  (red spot). The size of the red spot in this schematic  picture exemplifies  the magnitude of the total Burgers vector in the cluster. Right: the drilled domain $\Om_\e (\mu)$. Balls of radius $\e$, centred at the dislocation points $x_i$, are removed from $\Om$. A circulation condition on the strain is assigned on each $\de B_\e (x_i)$.}   
\label{fig core radius}
\end{figure} 
Under this assumption, a generic distribution of dislocations is identified 
with a measure
\[
\mu = \sum_{i=1}^N \xi_i \, \delta_{x_i} \,, \quad \xi_i \in \mathbb{S} \,,
\]
where
\[
|x_i-x_j| \geq 2 \rho_\e \,,  \quad  \dist(x_k,\de \Om) > \rho_\e \,, \quad \text{ for every } 1\leq i,j,k \leq N \,, \,\, i \neq j \,.
\] 
Denote by $\Om_\e (\mu) := \Om \setminus \bigcup_{i} B_\e (x_i)$ the drilled domain (see Figure \ref{fig core radius} Right).
The admissible strains associated to $\mu$ are matrix fields $\beta \in L^2 (\Om_\e (\mu); \matricidue)$ such that 
\[
\Curl \beta \zak \Om_\e (\mu)= 0
\] 
and
\begin{equation} \label{intro3 dis}
\int_{\de B_\e (x_i)} \beta \cdot t \, ds = \xi_i \,,  \quad \text{ for every } \quad i=1,\dots,N \,. 
\end{equation} 
The elastic energy corresponding to $(\mu,\beta)$ is defined by
\begin{equation} \label{intro3 en tot}
E_\e (\mu,\beta):= \int_{\Om_\e (\mu)} W(\beta) \, dx \,.
\end{equation}
The energy induced by the dislocation distribution $\mu$ is given by minimising \eqref{intro3 en tot} over the set of all strains satisfying \eqref{intro3 dis}. From  \eqref{intro3 dis} it follows that the energy is always positive if $\mu \neq 0$.
\par
The energy contribution of a single dislocation core is of order $|\log \e|$ (see Proposition \ref{paper3 prop log}). 
Therefore, for a system of $N_\e$ dislocations, with $N_\e \to \infty$ as $\e \to 0$, the relevant energy regime is
\[
E_\e  \approx N_\e |\log \e| \,.
\] 
This scaling was already studied in \cite{DGP} for  $N_\e\le C$.
The critical regime  $N_\e \approx |\log \e|$ has been considered for Ginzburg Landau vortices in \cite{jerrard} 
 and  for edge dislocations  in   \cite{glp}, where the authors, assuming that the dislocations are well separated,    characterise the $\Gamma$-limit of $\frac{E_\e}{|\log \e|^2}$. We will later discuss how this compares to our $\Gamma$-convergence result. 

For our analysis we will consider a higher energy regime corresponding to
\[
\frac{1}{\e} \gg N_\e \gg |\log \e| 
\]
(see Section \ref{paper3 setting} for the precise assumptions on $N_\e$). We will see that this energy regime will account for grain boundaries that are mutually rotated by  \textit{infinitesimal} angles $\theta \approx 0$. To be more specific, one can split the contribution of $E_\e$ into
\[
E_\e (\mu,\beta) =      E_\e^{\rm inter} (\mu,\beta) +  E_\e^{\rm self} (\mu,\beta) \,,
\]
where $E_\e^{\rm self}$ is the self-energy concentrated in the hard core region 
$\cup_i B_{\rho_\e}(x_i)$ while $E_\e^{\rm inter}$ is the interaction energy computed outside the hard core region. 
In Theorem \ref{thm:gamma} we will prove that the $\Gamma$-limit as $\e \to 0$ of 
the rescaled functionals $E_\e$, with respect to the strains and the dislocation measures, is of the form
\begin{equation} \label{paper3 intro gamma}
\mathcal{F} (\mu,S,A) = \int_\Om W(S) \, dx  +  \int_\Om  \f \left( \frac{d \mu}{d|\mu|} \right) \, d|\mu| \,.
\end{equation}
The first term of $\mathcal{F}$ comes from the interaction energy. It represents the elastic energy of the symmetric field $S$, which is the weak limit of the symmetric part of the strains rescaled by
$\sqrt{N_\e |\log \e|}$. Instead, the antisymmetric part of the strain, rescaled by $N_\e$, weakly converges to an antisymmetric field $A$. Therefore, since $N_\e \gg |\log \e|$, the symmetric part of the strain is of lower order with respect to the antisymmetric part.

The second term of $\mathcal{F}$ is the plastic energy. The density function $\f$ is positively $1$-homogeneous and it can be defined as the relaxation of a cell-problem formula (see Proposition \ref{paper3 prop log}). 
The measure $\mu$ in \eqref{paper3 gamma} is the weak-$*$ limit of the dislocation measures rescaled by $N_\e$, and $d\mu/d|\mu|$ represents the Radon-Nikodym derivative of $\mu$ with respect to $|\mu|$.
Notice that $A$ and $\mu$ come from the same rescaling $N_\e$, whereas the symmetric part $S$ is of lower order, namely $\sqrt{N_\e |\log \e|}$.
As a consequence, the compatibility condition \eqref{intro3 dis} passes to the limit as 
\[
\Curl A = \mu \,.
\]
This implies that the \textit{elastic} and \textit{plastic} terms in $\mathcal{F}$ are \textit{decoupled}. Indeed this is the main difference with the critical regime $N_\e \approx |\log \e|$ studied in \cite{glp}, where the contribution of the symmetric and antisymmetric part of the strain, as well as the dislocation measure, have the same order $|\log \e|$. This results, in \cite{glp}, into the coupling  of the two terms of the energy, through the condition $\Curl \beta = \mu$ where $\beta=S+A$.

Next we focus on the study of the $\Gamma$-limit $\mathcal{F}$. Precisely, we impose \textit{piecewise constant} Dirichlet boundary conditions on $A$, and we show that $\mathcal{F}$ is minimised by strains that are \textit{locally constant} and take values in the set of \textit{antisymmetric matrices}.  More precisely, 
there is a Caccioppoli partition of $\Om$ with sets of finite perimeter where the antisymmetric strain is constant. 
Such  sets represent the grains of the polycrystal, while the corresponding constant antisymmetric matrices represents their orientation. 
We call such configurations \textit{linearised polycrystals}. 
This definition is motivated by the fact that  antisymmetric matrices can be considered as infinitesimal rotations, being the linearisation about the identity of the set of rotations. 
   The proof of this result is based on the simple observation that the variational problem is equivalent to minimise some anisotropic total variation of a scalar function, which is locally constant on $\partial \Om$. By the coarea formula, one can easily show that there always exists a piece-wise constant minimiser.

The paper is organised as follows. In Section \ref{paper3 setting} we introduce the rigorous mathematical setting of the problem. 
In Section \ref{paper3 review glp} we recall some results from \cite{glp}, which will be useful for the $\Gamma$-convergence analysis of the rescaled energy $E_\e$. The main $\Gamma$-convergence result will be proved in Section \ref{paper3 gamma}. 
In Section \ref{paper3 boundary} we will include Dirichlet type boundary conditions to the $\Gamma$-convergence analysis performed in the previous section. Finally, in  Section \ref{paper3 poly} we will show that the plastic part of $\mathcal{F}$ is minimised by linearised polycrystals, by prescribing piecewise constant boundary conditions on the antisymmetric part of the limit strain.  

\section{Setting of the problem}\label{paper3 setting}

Let $\Omega \subset \R^2$ be a bounded open domain with Lipschitz continuous boundary. The set $\Om$ represents a horizontal section of an infinite cylindrical crystal $\Om \times \R$. Define as $\mathcal{S}:=\{b_1, \dots , b_s  \}$ the class of Burgers vectors. We will assume that $\mathcal{S}$ contains at least two linearly independent vectors so that
$\Span_\R \mathcal{S} = \R^2 $. 
We then define the set of slip directions
\[
\mathbb{S}:= \Span_\Z \mathcal{S} \,,
\] 
that coincides with the set of Burgers vectors for multiple dislocations. An edge dislocation can be identified with a point $x \in \Om$ and a vector $\xi \in \mathbb{S}$.

Let $\e >0$ be a parameter representing the interatomic distance of the crystal and denote by $\{N_\e\} \subset \N$   the number of dislocations present in the crystal at the scale $\e$. As in \cite{glp}, we introduce a hard core radius $\rho_\e \to 0$, and we assume that
\begin{enumerate}[(i)]
\item $\lim_{\e \to 0} \rho_\e / \e^s = + \infty$ for every fixed $0<s<1$\,,
\item $\lim_{\e \to 0} N_\e \rho_\e^2  = 0$\,, 	
\item $\lim_{\e \to 0} \frac{N_\e}{|\log \e|}   = +\infty$. 	
\end{enumerate}
The first condition implies that the hard core region contains almost all the self energy (see Proposition \ref{prop glp}), the second one guarantees that the area of the hard core region tends to zero, while the third one corresponds to the supercritical regime, where the interaction energy is dominant with respect to the self energy. The above conditions are compatible if
 \begin{equation}\label{compatibility}
\rho_\e = \e^{t(\e)}, \qquad N_\e = \e^{-t(\e)}  
 \end{equation}
 for some positive  $t(\e)$ converging to zero slowly enough (for instance such that $t(\e)|\log \e|<<\log(|\log \e|)$).
 The class of admissible dislocations is defined by
 \begin{equation} \label{AD}
 \begin{aligned}
\AD_\e (\Om) := \Big\{  \mu \in & \misure \, \colon \, \mu   = \sum_{i=1}^M \xi_i  \delta_{x_i} \, , \, M \in \N , \, \xi_i \in \mathbb{S} \,,  \\
  & B_{\rho_\e} (x_i) \subset \Om \,, \,\, |x_j - x_k| \geq 2 \rho_\e \,, \,\, \text{for every } i \text{ and } \, j \neq k \Big\} \,. 	
 \end{aligned}
\end{equation}
Here $\misure$ denotes the space of $\R^2$ valued Radon measures on $\Omega$ and $B_{r} (x)$ is the ball of radius $r$ centred at $x \in \R^2$. 

Fix a dislocations measure $\mu = \sum_{i=1}^M \xi_i  \delta_{x_i} \in \AD_\e (\Omega)$. For $r>0$ define
\begin{equation} \label{dominio perf}
\Omega_r (\mu) := \Om \setminus \cup_{i=1}^M \overline{B_r (x_i)} \,.	
\end{equation}
The class of admissible strains associated with $\mu$ is given by the maps $\beta \in L^2 (\Om_\e (\mu); \R^2)$ such that 
\[
\Curl \beta \zak \Omega_\e (\mu)= 0 \,, \quad  \quad \cir \beta \cdot t \, ds = \xi_i  \quad \text{for every} \quad i=1,\dots,M \,.
\]
The identity $\Curl \beta = 0$ is intended in the sense of distributions, where 
$\Curl \beta \in \distr(\Om;\R^2)$ is defined as
\begin{equation} \label{paper3 Curl def}
\Curl \beta := (\de_1 \beta_{12} - \de_2 \beta_{11},\de_1 \beta_{22} - \de_2 \beta_{21} ) \,.
\end{equation}
The integrand $\beta \cdot t$ is intended in the sense of traces, since $\beta \in H(\Curl,\Om_\e (\mu))$ (see \cite[Theorem 2, p. 204]{dautray}), 
and $t$ is the unit tangent vector to $\de B_\e (x_i)$, obtained by a counter-clockwise rotation of $\pi/2$ of the outer normal $\nu$ to $B_\e (x)$, that is $t :=J \nu$ with
\begin{equation} \label{paper3 J}
J:= \left( 
\begin{matrix}
0  &   -1 \\
1  &   0 	
\end{matrix}
\right) \,.
\end{equation}
In the following it will be useful to extend the admissible strains to the whole $\Om$. Therefore, for a dislocation measure $\mu= \sum_{i=1}^M \xi_i  \delta_{x_i} \in \AD_\e (\Omega)$, we introduce the class $\AS_\e (\mu)$ of admissible strains as
\begin{equation} \label{AS}
\begin{aligned}
\AS_\e (\mu) := & \Big\{ \beta \in \strain \, \colon \, \beta \equiv 0 \,\, \text{in} \,\, \Omega \setminus \Omega_\e (\mu) \,, \,\, \Curl \beta =0 \,\, \text{in} \,\, \Omega_\e (\mu)\,, \\ 
&\cir \beta \cdot t \, ds = \xi_i \,,\,\, \int_{\Om_\e (\mu)} \sk{\beta} \, dx = 0 \,, \,\, \text{for every} \,\,	i = 1,\dots,M \Big\} \,.	
\end{aligned}
\end{equation}
Here $\sk{F}:=  (F-F^T)/2$. The last condition in \eqref{AS} is not restrictive and will guarantee the uniqueness of the minimising strain.

The energy associated to an admissible pair $(\mu,\beta)$ with $\mu \in \AD_\e (\Omega)$ and $\beta~\in~\AS_\e (\mu)$ is defined by
\begin{equation}\label{Ee}
E_\e (\mu,\beta) := \int_{\Omega_\e (\mu)} W(\beta) \, dx=\int_{\Omega} W(\beta) \, dx \,,
\end{equation}
where
\[
W(F) :=\frac{1}{2} \C F \colon F =\frac{1}{2} \C \sym{F} \colon \sym{F}
\]
is the strain energy density, where   $\C$ is the elasticity tensor satisfying   
\begin{equation} \label{coercivity}
c^{-1} |\sym{F}|^2 \leq W(F) \leq c |\sym{F}|^2  \qquad \text{for every} \quad F \in \matricidue \,, 
\end{equation}
for some given  $c>0$. Since the elasticity tensor also satisfies the symmetry properties $\C_{ijkl} = \C_{klij} = \C_{ijlk} = \C_{jikl}$ (see \cite{bacon}), it follows that 
$$
\frac{1}{2} \C F \colon F =\frac{1}{2} \C \sym{F} \colon \sym{F} \,.
$$
Notice that for any $\mu \in \AD_\e (\Omega)$ the minimum problem
\begin{equation} \label{paper3 intro min}
\min \left\{ \int_{\Om_\e (\mu)} W (\beta) \, dx \, \colon \, \beta \in \AS_\e (\mu) \right\}
\end{equation}
has a unique solution. This can be seen by removing a finite number of cuts $L$ from $\Om_\e (\mu)$ so that $\Om_\e (\mu) \setminus L$ becomes simply connected and observing that there exists a displacement gradient such that $\nabla u = \beta$ in $\Om_\e (\mu) \setminus L$. Then we can apply the classic Korn inequality (see, e.g., \cite{ciarlet}) to $\nabla u$, and conclude by using the direct method of calculus of variations.

We recall that in our analysis we assume  the supercritical regime 
\begin{equation} \label{super}
N_\e \gg |\log \e| \,.
\end{equation}
As already discussed, the relevant scaling for the asymptotic study of $E_\e$ is given by $N_\e |\log \e|$. Therefore we introduce the scaled energy functional defined on the space $\misure \times \strain$ as
\begin{equation} \label{energy}
\mathcal{F}_\e (\mu, \beta) :=
\begin{cases}
 \displaystyle\frac{1}{N_\e |\log \e |} \, E_\e (\mu, \beta)   & \qquad \text{if} \,\, \mu \in \AD_\e (\Omega) \,,\,\, \beta \in \AS_\e (\mu) \,, \\
 + \infty                                         & \qquad \text{otherwise.}
 \end{cases}
\end{equation}

\section{Preliminaries} \label{paper3 review glp}

In this section we will recall some results and notations from \cite{glp} that will be needed in the following $\Gamma$-convergence analysis.

\subsection{Cell formula for the self-energy} \label{paper3 cell}

In this section we will rigorously define the density function $\f$ appearing in the $\Gamma$-limit $\mathcal{F}$ introduced in \eqref{paper3 intro gamma}. In order to do so, following \cite[Section 4]{glp}, we will introduce the self-energy $\psi(\xi)$ stored in the core region of a single dislocation $\xi \, \delta_0$ centred at the origin. 

Let us start by defining, for every $\xi \in \R^2$ and $0<r_1<r_2$, the space
\begin{equation} \label{self test}
\AS_{r_1,r_2} (\xi) := \left\{  \beta \in L^2 (B_{r_2} \setminus B_{r_1}; \matricidue) \, \colon \, \Curl \beta =0 ,\, 
\int_{\de B_{r_1}} \beta \cdot t \, ds = \xi    \right\}\,,	
\end{equation}
where $B_r$ is the ball of radius $r$ centred at the origin.
For strains belonging to such class, we have the following bound from below of the energy (see \cite[Remark 3]{glp}).


\begin{proposition}\label{paper3 prop bound}
Let $0<r_1<\frac{r_2}{2}$ and $\xi \in \R$. There exists a constant $c>0$ such that, for every $\beta \in AS_{r_1,r_2} (\xi)$, 
\begin{equation} \label{paper3 apriori}
\int_{B_{r_2} \setminus B_{r_1}} |\sym{\beta}|^2 \, dx \geq c |\xi|^2 \log \frac{r_2}{r_1} \,.
\end{equation}
\end{proposition}

Let $C_\e := B_1 \setminus B_\e$, with $0<\e<1$, and introduce $\psi_\e \colon \R^2 \to \R$ through the cell problem
\begin{equation} 
\psi_\e (\xi):=  \frac{1}{|\log \e|} \min \left\{   \int_{C_\e} W(\beta) \, dx \, \colon \,	\beta \in \AS_{\e,1} (\xi) \right\} \,.\label{psi eps}
\end{equation}
It is easy to show that the minimum in \eqref{psi eps} exists, by combining the classic Korn inequality with the direct method of the calculus of variations.
It is also immediate to check that the minimiser $\beta_\e (\xi)$ of \eqref{psi eps} satisfies the boundary value problem
\[
\begin{cases}
\Divv \C \beta_\e (\xi) = 0 & \text{ in } C_\e, \\
\C \beta_{\e} (\xi) \cdot \nu = 0   &   \text{ on } \de C_\e ,	
\end{cases}
\]
where $\nu$ is the outer normal to $\de C_\e$. Also,  there exists a  strain $\beta_{0}(\xi) \colon \R^2 \to \matricidue$ with $|\beta_{0}(\xi)(x)| \le c |x|^{-1} |\xi|$ (see \cite{bacon})   that is a distributional solution to 

\begin{equation}\label{dislor2}
\begin{cases}
\Divv \C \beta_0 (\xi) = 0 & \text{ in } \R^2, \\
\Curl \beta_0 (\xi) = \xi \, \delta_0    &   \text{ in } \R^2.	
\end{cases}
\end{equation}
The following result holds true (see \cite[Corollary 6]{glp}).

\begin{proposition}[Self-energy] \label{paper3 prop log}
There exists a constant $C>0$ such that for every $\xi \in \R^2$, 
\begin{equation} \label{paper3 stima psi eps}
\psi_\e (\xi ) \leq \frac{1}{|\log \e|}  \int_{C_\e} W(\beta_0(\xi)) \,dx \leq  \psi_\e (\xi) +  \frac{C|\xi|^2}{|\log \e|} \,.
\end{equation}
In particular, for every $\xi \in \R^2$, we have that 
\[
\lim_{\e \to 0} \psi_\e (\xi)= \psi (\xi) \,,
\] 
pointwise, where the map $\psi \colon \R^2 \to \R$ is the self-energy defined by 
\begin{equation} \label{psi}
\psi(\xi) := \lim_{\e \to 0} \frac{1}{|\log \e|} \int_{C_\e} W(\beta_0(\xi)) \,dx \,.
\end{equation}
Moreover, there exists a constant $c>0$ such that, for every $\xi \in \R^2$,
\begin{equation} \label{paper3 psi bound bis}
c^{-1} |\xi|^2 \leq \psi (\xi) \leq c |\xi|^2 \,.
\end{equation}
\end{proposition}

We now want to show that the self-energy $\psi(\xi)$ is indeed concentrated in the hardcore region $B_{\rho_\e} \setminus B_{\e}$ of the dislocation $\xi \, \delta_0$.
To this end, define the map $\bar{\psi}_\e \colon \R^2 \to \R$ as 
\begin{equation} 
\bar{\psi}_\e (\xi) :=\frac{1}{|\log \e|}  \min \left\{   \int_{B_{\rho_\e} \setminus B_\e} W(\beta) \, dx \, \colon \,	\beta \in \AS_{\e,\rho_\e} (\xi) 
\right\} \,,  \label{psi bar} 	
\end{equation}
for $\xi \in \R^2$. It will also be useful to introduce $\tilde{\psi}_\e \colon \R^2 \to \R$ as
\begin{equation}
	\tilde{\psi}_\e (\xi) :=\frac{1}{|\log \e|}  \min \left\{ \int_{B_{\rho_\e} \setminus B_\e} W(\beta) \, dx \, 
	\colon \, \beta \in \AS_{\e, \rho_\e} (\xi) , \, \beta \cdot t = \hat{\beta} \cdot t \, \text{ on } \, \de B_\e \cup \de B_{\rho_\e} \right\} \,, \label{psi tilde}
\end{equation}
where $\hat{\beta} \in \AS_{\e , \rho_\e } (\xi)$ is a fixed given strain such that 
\begin{equation} \label{crescita beta}
|\hat{\beta} (x)| \leq K \, \frac{|\xi|}{|x|} \,, 	
\end{equation}
for some positive constant $K$.  
By \eqref{coercivity}, it is immediate to see that problems \eqref{psi bar}-\eqref{psi tilde} are well posed.  
The following results holds (see \cite[Remark 7, Proposition 8]{glp}).

\begin{proposition} \label{prop glp}
We have $\bar{\psi}_\e (\xi) = \psi_\e(\xi) (1 + o (\e))$ and $\tilde{\psi}_\e(\xi) = \psi_\e (\xi)(1 + o (\e))$, with $o(\e) \to 0$ as $\e \to 0$  uniformly with respect to $\xi \in \R^2$. In particular 
\[
\lim_{\e \to 0} \bar{\psi}_\e (\xi) =\lim_{\e \to 0} \tilde{\psi}_\e (\xi) = \psi (\xi)
\]
pointwise, where $\psi$ is the self-energy defined in \eqref{psi}.
\end{proposition}

Now, we can define the density $\f \colon \R^2 \to [0,+\infty)$ as the relaxation of the self-energy $\psi$,
\begin{equation} \label{self}
\varphi(\xi):=	\inf \left\{  \sum_{k=1}^N   \lambda_k \psi (\xi_k) \, \colon \, \sum_{k=1}^N \lambda_k \xi_k = \xi , \, N \in \N, \, \lambda_k \geq 0, \, \xi_k \in \mathbb{S}  \right\} \,.
\end{equation}
The properties of $\varphi$ are summarised in the following proposition. 

\begin{proposition}\label{paper3 properties density}
The function $\f$ defined in \eqref{self} is convex and positively $1$-homogeneous, that is
\[
\f(\lambda \xi) = \lambda \f (\xi), \quad \text{for every } \xi \in \R^2, \, \lambda >0\,.
\]
Moreover there exists a constant $c>0$ such that 
\begin{equation} \label{self stima}
c^{-1} |\xi| \leq \f (\xi) \leq c |\xi| \,,
\end{equation}
for every $\xi \in \R^2$. In particular, the infimum in \eqref{self} is actually a minimum.
\end{proposition}

\subsection{Korn type inequality} \label{paper3 gen korn sec}

We will now recall the generalised Korn inequality 
proved in \cite[Theorem 11]{glp}.
\begin{theorem}[Generalised Korn inequality] \label{thm:korn}
There exists a constant $C>0$, depending only on $\Omega$, with the following property: for every $\beta \in L^1 (\Omega; \matricidue)$ with
\[
\Curl \beta = \mu \in \misure \,,
\]
we have
\begin{equation} \label{korn}
\int_{\Omega}  |  \beta -A   |^2 \, dx \leq C \left(    \int_{\Omega}  |\sym{\beta}|^2 \, dx + |\mu| (\Omega)^2 \right) \,, 	
\end{equation}
where $A$ is the constant $2 \times 2$ antisymmetric matrix defined by 
$A:= \frac{1}{|\Om|}\int_{\Omega} \sk{\beta} \, dx$.
\end{theorem}

\subsection{Remarks on the distributional Curl} \label{paper3 curl}

We conclude this section with some considerations on the distributional $\Curl$ of admissible strains (see \cite[Remark 1]{glp}).

\begin{remark}[Curl of admissible strains] \label{paper3 rem2}
Let $\mu \in \AD_\e (\Om)$ and $\beta \in \AS_\e (\mu)$. Recalling definition \eqref{paper3 Curl def}, we can define the scalar distribution
\[
\curl \beta_{(i)} := \frac{\de}{\de x_1} \beta_{i2} - \frac{\de}{\de x_2} \beta_{i1} \,,
\]
where $\beta_{(i)}$ denotes the $i$-th row of $\beta$. This means that for any test function $\f$ in $C^{\infty}_c (\Om)$, we can write
\begin{equation} \label{paper3 curl cond}
\langle \curl \beta_{(i)} , \f \rangle = - \int_\Om \beta_{(i)} \cdot J \nabla \f \, dx \,,
\end{equation}
where $J$ is the counter-clockwise rotation of $\pi/2$, as defined in \eqref{paper3 J}. Notice that, if $\beta_{(i)} \in L^2 (\Om;\R^2)$, then \eqref{paper3 curl cond} implies that $\curl \beta_{(i)}$  is well defined also for $\f \in H^1_0(\Om)$ and acts continuously on it. Therefore 
\[
\Curl \beta \in \duale  \quad \text{for every} \quad \beta \in \AS_\e (\mu) \,,
\]
where $\duale$ denotes the dual of the space $H^1_0 (\Om;\R^2)$.

Further, if $\mu = \sum_{i=1}^M \xi_i \, \delta_{x_i} \in \AD_\e (\Om)$, then the circulation condition
\[
\int_{\de B_\e (x_i)} \beta \cdot t \, ds = \xi_i  \,, \quad \text{for every} \quad i=1,\dots, M \,,
\]
can be written as 
\[
\langle \Curl \beta, \f \rangle = \sum_{i=1}^M \xi_i \, c_i \,,
\]
for every $\f \in H^1_0 (\Om)$ such that $\f \equiv c_i$ in $B_\e (x_i)$. If in addition $\f \in C^0 (\Om) \cap H^1_0 (\Om)$, then
\[
\langle \Curl  \beta, \f \rangle = \int_\Om \f \, d\mu \,.
\]
\end{remark}

\section{$\Gamma$-convergence analysis} \label{paper3 gamma}

In this section we will study, by means of $\Gamma$-convergence, the behaviour as $\e \to 0$ of the functionals $\mathcal{F}_\e \colon \misure \times \strain \to \R$ defined in \eqref{energy}, in the energy regime $N_\e \gg |\log \e|$. In Theorem \ref{thm:gamma} we will prove that the $\Gamma$-limit for the sequence $\mathcal{F}_\e$ is given by the functional
  $\mathcal{F} \colon (\misure \cap \duale) \times \lduesimm  \times \ldueanti \to \R$ defined as
\begin{equation} \label{limite}
\mathcal{F} (\mu, S,A):= 
\begin{cases}
 \displaystyle \int_\Omega W(S) \, dx + \int_\Omega \f \left( \displaystyle \frac{d \mu}{d |\mu|} \right) \, d |\mu|   & \text{if }  \, \Curl A = \mu,  \\
+ \infty    &  \text{otherwise} \,,
\end{cases}
\end{equation}
where $\f$ is the energy density introduced in \eqref{self}. 
The topology under which the $\Gamma$-convergence result holds is given by the following definition.

\begin{definition}\label{def:conv}
We say that the family (also referred to as sequence in the following)  $(\mu_\e,\beta_\e) \in \misure \times \strain$ is converging to a triplet $(\mu, S,A) \in \misure \times \lduesimm \times \ldueanti$ if 
\begin{gather}
\frac{\mu_\e}{N_\e}    \weakstar \mu   \quad \text{ in } \quad \misure \,,	\label{conv misure} \\ 
\frac{ \sym{\beta}_\e}{\sqrt{N_\e |\log \e|}   } \weak S  \qquad \text{and} \qquad \frac{\sk{\beta}_\e}{N_\e   }  \weak A   \quad \text{ weakly in } \quad \strain  \,. \label{conv strain} 
\end{gather}
\end{definition}

\begin{theorem}\label{thm:gamma} 
The following $\Gamma$-convergence result holds true. 

\begin{enumerate}[{(i)}]
\item \textnormal{(Compactness)} Let $\e_n \to 0$ and assume that $(\mu_n,\beta_n) \in \misure \times \strain$ is such that 
$\sup_n \mathcal{F}_{\e_n} (\mu_n,\beta_n ) \leq E$, for some positive constant $E$. 
Then there exists 
$$
(\mu,S,A) \in  (\misure \cap \duale ) \times \lduesimm \times \ldueanti,
$$ with $\Curl A = \mu$,
 such that up to subsequences (not relabelled),  
$(\mu_n,\beta_n)$ converges to $(\mu,S,A)$ in the sense of Definition \ref{def:conv}. 

\item \textnormal{(}$\Gamma$\textnormal{-convergence)} The functionals $\mathcal{F}_\e$ defined in \eqref{energy} $\Gamma$-converge to the functional $\mathcal{F}$ defined in \eqref{limite}, with respect to the convergence of Definition \ref{def:conv}. Specifically, for every
\[ 
 (\mu,S,A) \in (\misure \, \cap \, \duale) \times \lduesimm \times \ldueanti
\] 
such that $\Curl A = \mu$ we have: 
\begin{itemize}
\item \textnormal{(}$\Gamma$\textnormal{-liminf inequality)} 
For all sequences $(\mu_\e,\beta_\e) \in \misure \times \strain$ converging to $(\mu,S,A)$ in the sense of Definition \ref{def:conv}, 
\[
\mathcal{F} (\mu,S,A) \leq \liminf_{\e \to 0} \mathcal{F}_\e (\mu_\e,\beta_\e) \,.
\]
\item \textnormal{(}$\Gamma$\textnormal{-limsup inequality)}
	There exists a recovery sequence $(\mu_\e,\beta_\e)\in 
	\misure  \times \strain$, such that $(\mu_\e,\beta_\e)$ converges to $(\mu,S,A)$ in the sense of Definition \ref{def:conv}, and
\[
\limsup_{\e \to 0} \mathcal{F}_\e (\mu_\e,\beta_\e)  \leq \mathcal{F} (\mu,S,A)  \,.
\]
\end{itemize}
\end{enumerate}
\end{theorem}

\begin{remark}\label{umu}
Since $A$ is antisymmetric, there exist $u \in L^2(\Om)$ such that
\begin{equation}\label{au}
A=\left( \begin{matrix}
 0 &  u \\
 -u  &  0 \\	
 \end{matrix}
\right).
\end{equation}
Notice that $\Curl A = D u$.  Therefore,  $\Curl A  \in  \misure$ implies that $u \in BV(\Om)$ and curl $\mu=0$.
\end{remark}

\subsection{Compactness}

We will prove the compactness statement in Theorem \ref{thm:gamma}. Assume that $(\mu_n , \beta_n)$ is a sequence in $\misure \times \strain$ such that 
\begin{equation} \label{equibound}
\sup_n \mathcal{F}_{\e_n} (\mu_n,\beta_n ) \leq E \,.
\end{equation} 
The proof is divided into four parts.
\medskip

\noindent \textbf{Part 1.} Compactness of the rescaled measures. 
\medskip
 
\noindent Let $\mu_{n} := \sum_{i=1}^{M_n} \xi_{n,i}  \delta_{x_{n,i}} \in \AD_{\e_n} (\Om)$. We show that the total variation of $\mu_n/N_{\e_n}$ is uniformly bounded, i.e., there exists $C>0$ such that
\begin{equation} \label{paper3 tot bounded}
\frac{1}{N_{\e_n}} |\mu_n| (\Om) = \frac{1}{N_{\e_n}} \sum_{i=1}^{M_n} |\xi_{n,i}| \leq C \,,
\end{equation}
for every $n \in \N$.
Since the function $y \mapsto \beta_{n} (x_{n,i}+ y)$ belongs to
$\AS_{\e_n,\rho_{\e_n}} (\xi_{n,i})$, we have 
\begin{align*}
E  & \geq \mathcal{F}_{\e_n} (\mu_n,\beta_n) \geq  \frac{1}{N_{\e_n} |\log \e_n|} 
\sum_{i=1}^{M_n} \int_{ B_{\rho_{\e_n}} (x_{n,i}) \setminus B_{\e_n} (x_{n,i})    } W(\beta_n) \, dx \\
& =\frac{1}{N_{\e_n} |\log \e_n|} \sum_{i=1}^{M_n} \int_{ B_{\rho_{\e_n}} (0) \setminus B_{\e_n} (0)    } W(\beta_n (x_{n,i}+ y) ) \, dy 
\geq \frac{1}{N_{\e_n}} \sum_{i=1}^{M_n}  \bar{\psi}_{\e_n} (\xi_{n,i}) \,,
\end{align*}
where $\bar{\psi_\e}$ is defined in \eqref{psi bar}. Let $\psi$ be the self-energy in \eqref{psi} and set $c:= \frac{1}{2} \min_{|\xi|=1} \psi (\xi)$. Notice that $c>0$, by \eqref{paper3 psi bound bis}. By Proposition \ref{prop glp}, $\bar{\psi_\e} \to \psi$ pointwise as $\e \to 0$, therefore for sufficiently large $n$, we have $\bar{\psi}_{\e_n}(\xi) \geq c$ for every $\xi \in \R^2$ with $|\xi|=1$. Hence,
\begin{align*}
\frac{1}{N_{\e_n}} \sum_{i=1}^{M_n}   \bar{\psi}_{\e_n} (\xi_{n,i}) & = \frac{1}{N_{\e_n}} \sum_{i=1}^{M_n}  |\xi_{n,i}|^2 \, \bar{\psi}_{\e_n} \left( \frac{\xi_{n,i}}{|\xi_{n,i}|} \right) \geq  \frac{c}{N_{\e_n}} \sum_{i=1}^{M_n}  |\xi_{n,i}|^2  \\   
& \geq  \frac{c}{N_{\e_n}} \sum_{i=1}^{M_n}  |\xi_{n,i}| = c \, \frac{|\mu_n| (\Om)}{N_{\e_n}} \,.
\end{align*}
The last inequality follows from the fact that the vectors $\xi_{n,i}$ are bounded away from zero. By putting together the above estimates, we conclude that \eqref{paper3 tot bounded}, and in turn \eqref{conv misure} hold true.
\medskip

\noindent \textbf{Part 2.} Compactness of the rescaled $\sym{\beta_n}$. 
\medskip

\noindent This follows immediately by the bounds on the energy \eqref{coercivity}. Indeed by \eqref{equibound}, \eqref{Ee} and \eqref {coercivity},
\begin{equation} \label{bound sym}
C N_{\e_n} |\log \e_n| \geq C E_{\e_n} (\mu_n , \beta_n) \geq C \int_{\Om} |\sym{\beta}_n|^2 \, dx  \,,
\end{equation}
and the weak compactness of $ \sym{\beta}_n / \sqrt{ N_{\e_n} |\log \e_n| }$ in $\strain$ follows.
\medskip

\noindent \textbf{Part 3.} Compactness of the rescaled $\sk{\beta}_n$. 
\medskip

\noindent
Now that the bounds \eqref{paper3 tot bounded}-\eqref{bound sym} are established, the idea is to apply the generalised Korn inequality of Theorem \ref{thm:korn}, in order 
to obtain a uniform upper bound for $\sk{\beta}_n /N_{\e_n} $ in $\strain$. In order to do that, we need a control over $|\Curl \beta_n| (\Om)$. In fact, even if $\beta_n$ is related to $\mu_n$ by circulation compatibility conditions, the relationship between 
$|\Curl \beta_n| (\Om)$ and $|\mu_n|(\Om)$ has to be clarified. In order to obtain a bound for $|\Curl \beta_n| (\Om)$ in terms of $|\mu_n|(\Om)$, we will define new strains $\tilde{\beta}_n$ that have the same order of energy of $\beta_n$ and that satisfy 
$|\Curl \tilde{\beta}_n| (\Om)=| \mu_n | (\Om)$.

Recall that $\mu_n = \sum_{i=1}^{M_n} \xi_{i,n} \delta_{x_{i,n}}$. 
Define the annuli $C_{i,n}:= B_{2 \e_n} (x_{i,n}) \setminus B_{ \e_n} (x_{i,n})$ and the functions $K_{i,n} \colon C_{i,n} \to \matricidue$ by
\[
K_{i,n} (x) := \frac{1}{2 \pi} \xi_{i,n} \otimes J \frac{x-x_{i,n}}{|x-x_{i,n}|^2} \,,
\]
where $J$ is the counter-clockwise rotation of $\pi/2$.
It is immediate to check that
\[
	\int_{C_{i,n}} |K_{i,n}|^2 \, dx = C |\xi_{i,n}|^2 \,,
\]
where the constant $C>0$ does not depend on $\e_n$. By Proposition \ref{paper3 prop bound} we also have
\[
\int_{C_{i,n}} |\sym{\beta}_n|^2 \, dx \geq C |\xi_{i,n}|^2 \,,
\]
where, again, the constant $C>0$ does not depend on $\e_n$. Therefore 
\begin{equation} \label{stima K}
	\int_{C_{i,n}} |K_{i,n}|^2 \, dx \leq C \int_{C_{i,n}} |\sym{\beta}_n|^2 \, dx \,. 	
\end{equation}
Note that $\Curl K_{i,n}= \xi_{i,n} \delta_{x_{i,n}}$ in $\distr(\R^2;\R^2)$, hence $\Curl (\beta_n - K_{i,n})=0$ in $C_{i,n}$. Moreover
$\int_{\de B_{\e_n}(x_{i,n})} (\beta_n - K_{i,n}) \cdot t \, ds = 0$, therefore there exists $v_{i,n} \in H^1 (C_{i,n}; \R^2)$ such that
$\nabla v_{i,n}= \beta_n - K_{i,n}$ in $C_{i,n}$. By \eqref{stima K},
\[
\int_{C_{i,n}} | \nabla \sym{v}_{i,n}    |^2 \, dx \leq C \int_{C_{i,n}} |\sym{\beta}_n|^2 \, dx \,.
\]
By applying the classic Korn inequality we get
\[
\int_{C_{i,n}} | \nabla v_{i,n}  - A_{i,n}  |^2 \, dx   \leq C \int_{C_{i,n}} | \nabla \sym{v}_{i,n}    |^2 \, dx \leq C \int_{C_{i,n}} |\sym{\beta}_n|^2 \, dx \,,
\]
for some constant matrix $A_{i,n} \in \matricidue_{\text{skew}}$ and some constant $C>0$. By standard extension methods, there exists $u_{i,n} \in H^1 (B_{2\e_n} (x_{i,n}) ; \R^2   ) $ such that $\nabla u_{i,n} = \nabla v_{i,n} - A_{i,n}$ in $C_{i,n}$ and
\begin{equation} \label{stima u}
\int_{B_{2 \e_n} (x_{i,n})   } | \nabla u_{i,n}  |^2   \, dx \leq C   \int_{C_{i,n}} | \nabla v_{i,n}  - A_{i,n}  |^2 \, dx   \leq C \int_{C_{i,n}} |\sym{\beta}_n|^2 \, dx	\,.
\end{equation}
Define $\tilde{\beta}_n \colon \Omega \to \matricidue$ by setting
 \begin{equation}\label{defbt}
 \tilde{\beta}_n (x) :=  \begin{cases}
  \beta_n (x)  &  \qquad \text{if} \,\, x \in \Omega_{\e_n} (\mu_n) \,,\\
  	\nabla u_{i,n} (x) + A_{i,n}  &  \qquad \text{if} \,\, x \in B_{\e_n} (x_{i,n}) \,.
  \end{cases}
 \end{equation}
From \eqref{bound sym} and \eqref{stima u}, we have
\begin{equation*}
\begin{aligned}
\int_{\Om}   |\sym{ \tilde{\beta} }_n |^2 \, dx = &\int_{\Omega_{\e_n} (\mu_n)}  |\sym{\beta}_n |^2 \, dx + 
\sum_{i=1}^{M_n}  \int_{B_{\e_n} (x_{i,n})} |  \nabla \sym{u}_{i,n}     |^2 \, dx \\ 
& \leq C   \int_{\Om}   |\sym{\beta}_n |^2 \, dx \leq C N_{\e_n} |\log \e_n| \,.
\end{aligned}
\end{equation*}
Moreover by construction $\Curl \tilde{\beta}_n$ is concentrated on $\de B_{\e_n } (x_{i,n})$ and we have 
 \begin{equation}\label{defbt2}
|\Curl \tilde{\beta}_n | ( \overline B_{\e_n } (x_{i,n}))  = |\mu_n| (\overline B_{\e_n } (x_{i,n}))  \text{ for all } i, \qquad  |\Curl \tilde{\beta}_n | (\Omega ) = |\mu_n| (\Omega).
 \end{equation}
 Therefore we can apply the generalised Korn inequality of Theorem \ref{thm:korn} to get
\[
\begin{aligned}
\int_\Om |\tilde{\beta}_n - \tilde{A}_n|^2 \, dx & \leq C \left(   \int_\Om |\sym{\tilde{\beta}}_n  |^2 \, dx + (|\mu_n| (\Omega))^2 \right)  \\
& \leq C  \left(  N_{\e_n}   |\log \e_n|    + N_{\e_n}^2 \right) \leq C  N_{\e_n}^2 \,, 
\end{aligned} 
\]
where $\tilde{A}_n := \frac{1}{|\Omega|} \int_\Om \sk{\tilde{\beta}}_n \in \matricidue_{\rm skew}$. The last inequality follows from the assumption $ |\log \e_n| \ll N_{\e_n}$.
Now recall that by hypothesis the average of $\beta_n$ is a symmetric matrix and $\beta_n \equiv 0$ in $\Omega \setminus \Omega_{\e_n} (\mu_n)$. Therefore, since symmetric and skew matrices are orthogonal, we have
$|\beta_n - \tilde{A}_n|^2=|\beta_n|^2 + |\tilde{A}_n|^2$, so that
\[
\int_{\Omega_{\e_n} (\mu_n)}  |\beta_n|^2 \, dx \leq \int_{\Omega_{\e_n} (\mu_n)}  |\beta_n - \tilde{A}_n|^2 \, dx \leq
\int_{\Omega}  |\tilde{\beta}_n - \tilde{A}_n|^2 \, dx \leq C N_{\e_n}^2 \,, 
\] 
which yields the desired compactness property for $ \sk{\beta}_n /N_{\e_n}$ in $\strain$.
\medskip

\noindent \textbf{Part 4.} $\mu \in H^{-1} (\Omega;\R^2)$ and $\Curl A = \mu$.
\medskip

\noindent Recall that $\mu_n = \sum_{i=1}^{M_n} \xi_{n,i} \delta_{x_{n,i}} \in \AD_{\e_n} (\Om)$ and $\beta_n \in \AS_{\e_n}(\mu_n)$. Let $\f \in C^1_0(\Omega)$ and $\f_n \in H^1_0 (\Omega)$ be a sequence converging to $\f$ uniformly and
strongly in $H^1_0 (\Omega)$, and such that
\[
\f_n \equiv \f (x_{n,i}) \quad \text{ in } \quad B_{\e_n} (x_{n,i}) \,.
\] 
By Remark \ref{paper3 rem2}, we then have 
\[
\int_{\Om} \f_n \, d\mu_n = \langle \Curl \beta_n , \f_n \rangle = \int_{\Om} \beta_n J \nabla \f_n \,dx \,. 
\]
Hence, by invoking \eqref{super}, \eqref{conv misure} and \eqref{conv strain}, we have
\[
\begin{aligned}
\int_{\Omega} \f \, d\mu & = \lim_{n \to \infty} \frac{1}{N_{\e_n}} \int_{\Omega} \f_n \, d \mu_n	=  
\lim_{n \to \infty} \frac{1}{N_{\e_n}} \langle \Curl \beta_n , \f_n \rangle \\
& = \lim_{n \to \infty} \frac{1}{N_{\e_n}}  \int_{\Omega} \beta_n J \nabla \f_n \, dx = 
\int_{\Omega} A \, J \nabla \f \, dx = \langle \Curl A, \f \rangle \,.
\end{aligned}
\] 
From this we conclude that $\Curl A = \mu$. Moreover, since $A \in L^2(\Om;\matricidue)$, then by definition $\Curl A \in \duale$. Hence also $\mu \in H^{-1}(\Omega;\R^2)$.

\subsection{$\Gamma$-liminf inequality}

We now want to prove the $\Gamma$-liminf inequality of Theorem \ref{thm:gamma}. Let
$\mu_\e \in \AD_\e (\Omega)$, $\beta_\e \in \AS_\e (\mu_\e)$ and
\[
(\mu,S,A) \in (  \misure \, \cap \, \duale ) \times \lduesimm \times \ldueanti \,,
\]
such that $\Curl A = \mu$. 
Assume that $(\mu_\e,\beta_\e)$ converges to $(\mu,S,A)$ in the sense of Definition \ref{def:conv}. We have to show that
\begin{equation} \label{paper 3 liminf proof}
\liminf_{\e \to 0} \mathcal{F}_\e (\mu_\e,\beta_\e)  \geq \mathcal{F} (\mu,S,A)  \,.
\end{equation}
In order to do so, we decompose the energy in
\begin{equation} \label{tesi liminf}
\frac{1}{N_\e |\log \e|} \int_{\Omega} W(\beta_\e) \, dx = 
\frac{1}{N_\e |\log \e|} \int_{\Omega_{\rho_\e} (\mu_\e)} W(\beta_\e) \, dx +
\frac{1}{N_\e |\log \e|} \int_{\Omega \setminus \Omega_{\rho_\e} (\mu_\e)} W(\beta_\e) \, dx 
\end{equation}
and study the two contributions separately. 

Recall that $\mu_\e= \sum_{i=1}^{M_\e} \xi_{\e,i} \delta_{x_{\e,i}}$. Since we are assuming that $\mu_\e /N_\e \weakstar \mu$, this implies that $|\mu_\e|(\Om)/N_\e$ is uniformly bounded, hence
$M_\e \leq C N_\e$ for some uniform constant $C>0$. Moreover $N_\e \rho_\e^2 \to 0$ by hypothesis, therefore $\rchi_{\Omega_{\rho_\e}} \to 1$ in $L^1(\Omega)$, as
\[
\int_{\Om} |\rchi_{\Omega_{\rho_\e}} - 1| \, dx = \sum_{i=1}^{M_\e} |B_{\rho_\e} (x_{\e,i}) | = \pi \rho_\e^2 M_\e \leq C \rho_\e^2 N_\e\,. 
\] 
Since $\sym{\beta}_\e / \sqrt{N_\e |\log \e|} \weak S$, we deduce that 
\[
\frac{\sym{\beta}_\e \rchi_{\Omega_{\rho_\e}}}{ \sqrt{N_\e |\log \e|}} \weak S \quad \text{ weakly in } \quad L^2(\Omega; \matricidue) \,.
\] 
Hence, by weak lower semicontinuity,
\begin{equation} \label{paper3 liminf0}
\begin{aligned}
\liminf_{\e \to 0} \frac{1}{N_\e |\log \e|} 
\int_{\Omega_{\rho_\e} (\mu_\e)} W(\beta_\e) \, dx  & =
\liminf_{\e \to 0} \int_{\Omega} W \left(  \frac{ \sym{\beta}_\e \rchi_{ \Omega_{\rho_\e}}}{ \sqrt{N_\e |\log \e|}}   \right) \, dx \\   
 & \geq \int_\Omega W(S) \, dx \,. 
  \end{aligned}
\end{equation}
Let us consider the second integral in \eqref{tesi liminf}. By Proposition \ref{prop glp} we have
\begin{equation} \label{paper3 liminf1}
\begin{aligned}
\frac{1}{|\log \e|} \int_{\Omega \setminus \Omega_{\rho_\e} (\mu_\e)} W (\beta_\e) \, dx & = \sum_{i=1}^{M_\e}  \frac{1}{|\log \e|} \int_{B_{\rho_\e} (x_{\e,i})} W (\beta_\e) \, dx \\ 
& \geq   \sum_{i=1}^{M_\e} \bar{\psi}_\e (\xi_{\e,i})   
  =(1 + o (\e)) \sum_{i=1}^{M_\e} \psi (\xi_{\e,i}) \,,
\end{aligned}
\end{equation}
where $o(\e) \to 0$ as $\e \to 0$. By the properties of $\f$ (Proposition \ref{paper3 properties density}) and by Reshetnyak's lower semicontinuity Theorem (\cite[Theorem 2.38]{afp}),
\begin{equation} \label{paper3 liminf2}
\begin{aligned}
\liminf_{\e \to 0} \frac{1}{N_\e} \sum_{i=1}^{M_\e} \psi (\xi_{\e,i}) 
& \geq \liminf_{\e \to 0} \frac{1}{N_\e} \sum_{i=1}^{M_\e} \f (\xi_{\e,i}) \\
& = \liminf_{\e \to 0} \frac{1}{N_\e} \int_\Omega \f \left(\frac{d \mu_\e}{d | \mu_\e| }   \right)  \, d |\mu_\e |   
  \geq  \int_\Omega \f \left(\frac{d \mu}{d | \mu| }   \right) \, d |\mu | \, .   
\end{aligned}
\end{equation}
By \eqref{paper3 liminf1}-\eqref{paper3 liminf2} we get
\[
\liminf_{\e \to 0} \frac{1}{N_\e |\log \e|} \int_{\Omega \setminus \Omega_{\rho_\e} (\mu_\e)}  W(\beta_\e) \, dx   \geq  \int_\Omega \f \left(\frac{d \mu}{d | \mu| }   \right) \, d |\mu | \,,
\]
that together with \eqref{paper3 liminf0} yields the $\Gamma$-liminf inequality \eqref{paper 3 liminf proof}.

\subsection{$\Gamma$-limsup inequality}

In this section we prove the $\Gamma$-limsup inequality of Theorem \ref{thm:gamma}.
Before proceeding, we need the following technical lemma to construct the recovery sequence for the measure $\mu$. Let us first introduce some notation. For a sequence of atomic vector valued measures of the form $\nu_\e := \sum_{i=1}^{M_\e} \alpha_{\e,i} \delta_{x_{\e,i}}$ and a sequence $r_\e \to 0$, we define the corresponding diffused measures
\begin{equation} \label{mis diffuse}
\tilde{\nu}^{r_\e}_\e := \frac{1}{\pi r_\e^2} \sum_{i=1}^{M_\e} \alpha_{\e, i} \, \mathcal{H}^2 \zak  B_{r_\e} (x_{\e,i}) 
  \,,  \quad 
\hat{\nu}^{r_\e}_\e := \frac{1}{2 \pi r_\e} \sum_{i=1}^{M_\e} \alpha_{\e, i} \, \mathcal{H}^1 \zak \de B_{r_\e} (x_{\e, i}) \,. 
\end{equation}
For $x_{\e,i} \in \supp \nu_\e$, define the functions $\tilde{K}_{\e,i}^{\alpha_{\e,i}}, \, \hat{K}_{\e,i}^{\alpha_{\e,i}} \colon B_{r_\e} (x_{\e,i}) \to \matricidue$ as
\begin{equation} \label{K diffusi}
\tilde{K}_{\e,i}^{\alpha_{\e,i}} (x) := \frac{1}{2 \pi r_\e^2} \, \alpha_{\e,i} \otimes J (x-x_{\e,i}) \,, \quad
\hat{K}_{\e,i}^{\alpha_{\e,i}} (x) := \frac{1}{2 \pi } \, \alpha_{\e,i} \otimes J \frac{x-x_{\e,i}}{|x-x_{\e,i}|^2} \,,
\end{equation}
where $J$ is the counter-clockwise rotation of $\pi/2$. 
Finally define $\tilde{K}_\e^{\nu_\e} , \, \hat{K}_\e^{\nu_\e} \colon \Omega \to \matricidue$ as
\begin{equation} \label{K diffusi bis}
\tilde{K}_\e^{\nu_\e}:= \sum_{i=1}^{M_\e} \tilde{K}_{\e,i}^{\alpha_{\e,i}} \rchi_{B_{r_\e} (x_{\e,i})} \, , \quad  	
\hat{K}_\e^{\nu_\e}:= \sum_{i=1}^{M_\e} \hat{K}_{\e,i}^{\alpha_{\e,i}} \rchi_{B_{r_\e} (x_{\e,i})} \,.
\end{equation}
It is easy to show that
\begin{equation} \label{curl misure}
\Curl \tilde{K}_\e^{\nu_\e} = \tilde{\nu}_\e^{r_\e} - \hat{\nu}_\e^{r_\e} \,, \quad 
\Curl \hat{K}_\e^{\nu_\e} = \nu_\e - \hat{\nu}_\e^{r_\e}	\,.
\end{equation}

The following easy lemma will be used in some density argument in the construction of the recovery sequence.
\begin{lemma} \label{lemma mis0}
Let $n\in\N$, and set 
\begin{equation}\label{spexi}
S_n:= \left\{\xi := \sum_{k=1}^M \lambda_k \xi_k \text{ with } M\in \N, \, \xi_k \in \mathbb{S}, \lambda_k>0 \text{ such that } z_j:=\frac{n^2 \lambda_j}{\sum \lambda_k} \in\N  \text{ for all } j \right\}.
\end{equation}
The union of such sets is dense in $\R^2$.
\end{lemma}
\begin{lemma} \label{lemma mis}
Let conditions (i), (ii) and (iii) of Section \ref{paper3 setting} hold true. Then we have:
\begin{itemize}
\item[(A)]
Let $n\in\N$, $\xi\in S_n$ defined as in \eqref{spexi} and let $\mu:= \xi \, dx$. Set $\Lambda:= \sum_{k=1}^M \lambda_k$,
$r_\e := \frac{1}{ 2  \sqrt{\Lambda N_\e}}$.
Then,  there exists a sequence  $\eta_\e = \sum_{k=1}^M \xi_k \eta_\e^k$, with $\eta_\e^k = \sum_{l=1}^{M_\e^k} \delta_{x_{\e,l}}$, such that $\eta_\e \in \AD_\e(\Om)$, and
\begin{gather}
\frac{|\eta_\e^k|}{N_\e} \weakstar \lambda_k \, dx \quad \text{in} \quad \mathcal{M}(\Om;\R) \,, \qquad \frac{\eta_\e}{N_\e} \weakstar \mu \quad \text{in} \quad \misure \,,   
\label{paper3 conv1}	 \\
 \left \| \frac{\tilde{\eta}^{r_\e}_\e}{N_\e} - \mu \right\|_{H^{-1}(\Om;\R^2)} \le \frac{n C}{ \sqrt  {N_\e} },
\label{conv2}
\end{gather}
for some constant $C$ independent of $n$, where the measure $\tilde{\eta}^{r_\e}_\e$ is defined according to \eqref{mis diffuse}.
\item[(B)]
  Let $\mu$, $r_\e$ as in $(A)$,   let $g \in C^0 (\overline \Om;  \R^2)$  and set $\sigma:= g(x) \, dx$. 
  Then, there exists a sequence $\eta_\e$ satisfying all the properties in $(A)$ and a sequence 
    $\sigma_\e = \sum_{l=1}^{H_\e} \zeta_{\e,l} \delta_{y_{\e,l}}$, with $\zeta_{\e,l} \in \mathbb S$, such that 
      supp$(\sigma_\e)\cap$ supp$(\eta_\e) = \emptyset$,  $\eta_\e + \sigma_\e \in \AD_\e(\Om)$
and  
\begin{gather}
\label{conv3}
\frac{\sigma_\e}{\sqrt{N_\e |\log \e|}} \weakstar \sigma \quad \text{in} \quad \misure \,, \qquad 
\frac{\tilde \sigma_\e}{\sqrt{N_\e |\log \e|}} \to \sigma  \quad \text{in} \quad \duale,
\end{gather}
where the measures $\tilde{\sigma}^{r_\e}_\e$ are defined according to \eqref{mis diffuse}.
\end{itemize}
In particular there exists a constant $C>0$ such that
\begin{equation} \label{bound masse}
 H_\e \leq C \sqrt{N_\e |\log \e|} \,, \qquad M_\e \leq C N_\e   \,,	
\end{equation}
where $M_\e:=\sum_{k=1}^M M_\e^k$.
\end{lemma}

\begin{proof}

\medskip
\noindent \textbf{Step 1.}  {\it Proof of (A), the case $M=1$ and $\mu=\xi \, dx$ with $\xi \in \mathbb{S}$}.
\medskip

\begin{figure}[t!] 
\centering   
\def\svgwidth{8.5cm}   
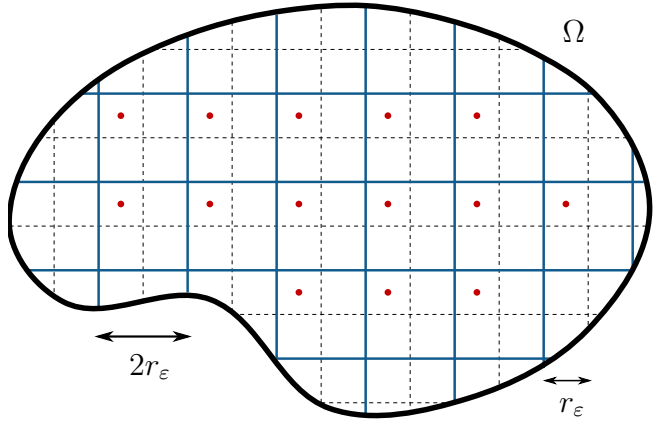  
\caption{Approximating $\mu = \xi \, dx$ with the $2r_\e$-periodic atomic measure $\eta_\e$. The red dots represent Dirac masses $\xi \delta_{x_{\e,i}}$ in the support of $\eta_\e$.}   
\label{fig recovery}
\end{figure} 

\noindent We cover $\R^2$ with squares of side length $2 r_\e$. Divide  each of them  in   four squares of side length $r_\e$,  and plug a mass $\xi \,\delta_{x_{\e,i}}$ at the centre of  one  of such $r_\e$-squares, obtaining  in this way a measure $\nu_\e$ on $\R^2$ which is $2 r_\e$ periodic. We notice that we leave some free space just in order to accomplish also point $(B)$.   
 Then we define $\eta_\e$ as the restriction of $\nu_\e$ on all the $2 r_\e$-squares contained in $\Om$ (see Figure \ref{fig recovery}). Notice that 
 $\eta_\e \in \AD_\e (\Om)$ since $r_\e \gg 2 \rho_\e$. 
Also, the density of $\frac{1}{N_\e} \tilde{\eta}^{r_\e}_\e - \mu$ has zero average on each $2 r_\e$-square, so that it  converges to zero weakly in $L^2(\Om;\R^2)$ and \eqref{paper3 conv1} is verified.

Let $v_\e:\R^2\to \R^2$ be the $2 r_\e$-periodic solution to  $\Delta v_\e = \frac{1}{N_\e} \tilde\nu_\e^{r_\e} - \mu$.
By construction it is easy to see  that
\begin{equation}\label{permenouno}
\left\|  \frac{1}{N_\e} \tilde\nu_\e^{r_\e} - \mu \right \|_{H^{-1}(\Om)} \le  \|v_\e\|_{H^1(\Om;\R^2)} \le C r_\e, \qquad \left\| \frac{1}{N_\e} \tilde\eta_\e^{r_\e} - \frac{1}{N_\e} \tilde\nu_\e^{r_\e} \right \|_{H^{-1}(\Om;\R^2)} \le C r_\e.
\end{equation}
These last estimates  clearly imply  \eqref{conv2}.

\medskip
\noindent \textbf{Step 2.} {\it Proof of $(A)$, the general case $\xi\in S_n$}.
\medskip

\noindent  Cover $\R^2$ with squares of side length $2 n r_\e$, and divide  each of them  in   four squares of side length $n r_\e$. 
As in Step 1, pick one  of these $n r_\e$-squares  in all $2n r_\e$-squares in a periodic manner. Finally, divide 
  each of these selected $n r_\e$-squares in $n^2$ squares of side length $ r_\e$. Now, plug at the centres of each of these 
  $n^2$ squares a mass $\xi_k \,\delta_{x_{\e,i}}$ with $1\le k\le M$, in such a way that the resulting measure $\nu_\e$ is $2n r_\e$-periodic, and on each 
  $2n r_\e$-square there are exactly $z_k$ masses with weight $\xi_k$, where $z_k$ is defined in \eqref{spexi}.
Then, defining $\eta_\e$ as the restriction of $\nu_\e$ on the union of all $2n r_\e$-squares contained in $\Om$, and arguing as in the proof of Step 1, we have that \eqref{paper3 conv1} holds true, while \eqref{permenouno} holds true with $C$ replaced by $n C$, so that \eqref{conv2} follows.

\medskip
\noindent \textbf{Step 3.} {\it Proof of $(B)$}.
\medskip

\noindent We have at disposal $C N_\e$ squares of side length $ n r_\e$, left free from the constructions in Step 2. Clearly, we can plug masses with weights in $\mathbb S$ at the centre of $c \sqrt{N_\e |\log \e|}$ of such free squares, in such a way that \eqref{conv3} holds true.  
%
%
%
\end{proof}

We are now ready to prove the $\Gamma$-limsup inequality of Theorem \ref{thm:gamma}.

\begin{proof}[Proof of $\Gamma$-limsup inequality of Theorem \ref{thm:gamma}] 

Let
\[
(\mu,S,A) \in (\misure \, \cap \, \duale) \times \lduesimm \times \ldueanti \,,
\]
with $\Curl A = \mu$. We will construct a recovery sequence in three steps. 
\medskip

\noindent \textbf{Step 1.} The case $\mu = \xi \, dx$ with  $S \in C^1 (\overline \Omega;\matricidue_{\text{sym}}$).
\medskip

\noindent In this step we assume that $\mu := \xi \, dx$, $A \in \ldueanti$ with $\Curl A = \mu$
 and  $S \in C^1 (\overline\Omega;\matricidue_{\text{sym}})$. We will construct a recovery sequence $\mu_\e \in \AD_\e (\Omega)$, $\beta_\e \in \AS_\e (\mu_\e)$, such that $(\mu_\e,\beta_\e)$ converges to $(\mu,S,A)$ in the sense of Definition \ref{def:conv} and 
\begin{equation} \label{thesis}
\limsup_{\e \to 0} \frac{1}{N_\e |\log \e|} \int_{\Omega}  W(\beta_\e) \, dx \leq \int_{\Omega} ( W(S) + \f (\xi) ) \, dx \,. 	
\end{equation}

By Proposition \ref{paper3 properties density}, there exist $\lambda_k \geq 0$, $\xi_k \in \mathbb{S}$, $M \in \N$, such that $\xi = \sum_{k=1}^M \lambda_k \xi_k$ and
\begin{equation} \label{min self}
\varphi(\xi) = \sum_{k=1}^M \lambda_k \psi (\xi_k) \,,	
\end{equation}
where $\varphi$ is the self-energy defined in \eqref{self}.
 By standard density arguments in $\Gamma$-convergence, we will assume without loss of generality that $\xi \in S_n$ is as in \eqref{spexi}  for some $n\in\N$.
  
Set $\sigma:=\Curl S$. Since $S \in C^1 (\overline \Omega; \matricidue_{\text{sym}})$, then $\sigma= g(x) \, dx$ for some continuous function $g \colon \overline \Omega \to \R^2$. 
Let $\eta_\e:= \sum_{i=1}^{M_\e} \xi_{\e,i} \delta_{x_{\e,i}}$, $\sigma_\e := \sum_{i=1}^{H_\e} \zeta_{\e,i} \delta_{y_{\e,i}}$ and $r_\e:=C/\sqrt{N_\e}$ be the sequences given by Lemma \ref{lemma mis} (B). 
Set $\mu_\e:= \eta_\e + \sigma_\e$. By \eqref{paper3 conv1}, \eqref{conv3} and the hypothesis $N_\e \gg |\log \e|$, $\mu_\e$ is a recovery sequence for $\mu$.

Let $\tilde{\eta}_\e^{r_\e}, \hat{\eta}_\e^{r_\e}, \tilde{\sigma}_\e^{r_\e},\hat{\sigma}_\e^{r_\e}$ be defined according to \eqref{mis diffuse}. 
Notice that $\hat{K}_{\e,i}^{\xi_{\e,i}} \in \AS_{\e, \rho_\e} (\xi_{\e,i})$ and it satisfies \eqref{crescita beta}. Therefore, by Proposition \ref{prop glp}, there exist strains $\hat{A}_{\e,i}$ such that
\begin{enumerate}[(i)]
\item 	$\hat{A}_{\e,i} \in \AS_{\e, \rho_\e} (\xi_{\e,i})$,
\item   $\hat{A}_{\e,i} \cdot t = \hat{K}_{\e,i}^{\xi_{\e,i}} \cdot t$ on $\de B_\e (x_{\e,i}) \cup \de B_{\rho_\e (x_{\e,i})}$, 
\end{enumerate}
and
\begin{equation}
\frac{1}{|\log \e|} \int_{B_{\rho_\e} (x_{\e,i}) \setminus B_{\e} (x_{\e,i})} W( \hat{A}_{\e,i} ) \, dx = \psi(\xi_{\e,i})(1+ o (\e)) \label{A eps}.  	
\end{equation}
Now extend $\hat{A}_{\e,i}$ to be $\hat{K}_{\e,i}^{\xi_{\e,i}}$ in 
$B_{r_\e} (x_{\e,i}) \setminus B_{\rho_\e} (x_{\e,i})$ and zero in $\Omega \setminus ( B_{r_\e} (x_{\e,i}) \setminus B_{\e} (x_{\e,i})   )$. 
Set
\begin{equation} \label{nuclei}
\hat{S}_\e := \sum_{l=1}^{H_\e} \hat{K}^{\zeta_{\e,i}}_\e \, \rchi_{B_{r_\e} (y_{\e,i}) \setminus B_{\e} (y_{\e,i})} \,, \quad	\hat{A}_\e := \sum_{i=1}^{M_\e} \hat{A}_{\e,i} \,.
\end{equation}
Hence, recalling definition \eqref{mis diffuse} we have
\begin{equation} \label{nuclei curl} 
\Curl \hat{S}_{\e} = - \hat{\sigma}^{r_\e}_\e  + \hat{\sigma}^{\e}_\e \,,  \quad    
\Curl \hat{A}_{\e}= - \hat{\eta}^{r_\e}_\e + \hat{\eta}^{\e}_\e \,. 
\end{equation}

Define $Q_\e := J \, \nabla u_\e $,  $R_\e := J \, \nabla v_\e $ where $u_\e, \, v_\e$ solve 
\begin{equation} \label{laplace 1}
	\begin{cases}
	 \Delta u_\e =  	 \tilde{\sigma}_\e^{r_\e} - \sqrt{N_\e | \log \e |} \, \sigma \quad \text{ in } \Om \\
	\frac{u_\e}{\partial \nu} = C_{u,\e} \qquad \text{ on } \partial \Om;
	\end{cases}
	,
	\qquad \begin{cases}
	 \Delta v_\e =  	\tilde{\eta}_\e^{r_\e}  - N_\e  \mu \quad \text{ in } \Om \\
  \frac{v_\e}{\partial \nu} = C_{v,\e} \qquad \text{ on } \partial \Om\, ,
	\end{cases}
\end{equation}
where the constants $C_{u,\e}, \, C_{v,\e}$ are satisfy the compatibility condition
$$
\int_{\partial \Om} C_{u,\e} \, ds = \int_\Om \tilde{\sigma}_\e^{r_\e} - \sqrt{N_\e | \log \e |} \, dx, 
\qquad 
\int_{\partial \Om} C_{v,\e} \, ds = \int_\Om \tilde{\eta}_\e^{r_\e}  - N_\e  \mu \, dx.
$$
In this way,
\begin{equation} \label{curl Q}
\Curl Q_\e = \tilde{\sigma}_\e^{r_\e} - \sqrt{N_\e | \log \e |} \, \sigma, \qquad \Curl R_\e =\tilde{\eta}_\e^{r_\e}  - N_\e  \mu  \,.
\end{equation}
Notice that by construction $\frac{1}{\sqrt{N_\e|\log\e|}} (|C_{u,\e}| + | C_{v,\e}| \to 0$ as $\e\to 0$. 
Therefore, using also  \eqref{conv2}, \eqref{conv3} and standard elliptic estimates, we have
\begin{equation} \label{stima Q}
\frac{Q_\e}{ \sqrt{N_\e |\log \e|}}    \to 0, \quad \frac{R_\e}{ \sqrt{N_\e |\log \e|}}    \to 0   \qquad \text{ in } \quad L^2 (\Omega;\matricidue )\,. 	
\end{equation}
Also notice that
\begin{equation} \label{de om}
\frac{Q_\e + R_\e}{\sqrt{N_\e |\log \e|}} \cdot t \to 0    \quad \text{in } \,\, H^{-1/2}(\de \Om;\R^2) \cap L^1(\de \Om;\R^2) \,.	
\end{equation}

We can now define the candidate recovery sequence as
\begin{equation} \label{rec seq strain}
\mu_\e = \eta_\e + \sigma_\e,    \qquad \beta_\e := ( S_\e   +   A_\e ) \, \rchi_{\Omega_\e (\mu_\e)} \,,
\end{equation}
where
\begin{gather}\label{seq1}
S_\e := \sqrt{N_\e \va{ \log \e }} \, S +    \hat{S}_{\e} -  \tilde{K}_{\e}^{\sigma_\e} + Q_\e \,,   \\
\label{seq2}
A_\e := N_\e A +  \hat{A}_{\e} -  \tilde{K}_{\e}^{\eta_\e} + R_\e \,.     
\end{gather}
By definition and \eqref{curl misure}, \eqref{nuclei curl}, \eqref{curl Q}, it is immediate to check that
\begin{equation*}
\Curl S_\e =  \hat{\sigma}_{\e}^{\e} \,,\quad \Curl A_\e = \hat{\eta}^{\e}_{\e} \quad \text{ in } \,\, \Om \,.
\end{equation*}
Recalling that  $\mu_\e = \eta_\e + \sigma_\e$, we deduce that
\begin{equation*}
\Curl \beta_\e = \hat{\eta}^{\e}_\e  + \hat{\sigma}^{\e}_\e = \hat{\mu}^{\e}_\e \quad \text{ in } \,\, \Om \,, \qquad 
\Curl \beta_\e \zak \Omega_\e (\mu_\e) =0 .
\end{equation*}
Moreover,  the circulation condition $\int_{\de B_\e (x)} \beta_\e \cdot t \, ds = \mu_\e(x)$ is satisfied for every point $x \in \spt \mu_\e$. Hence $\beta_\e \in \AS_\e (\mu_\e)$. 

In order for $(\mu_\e,\beta_\e)$ to be the desired recovery sequence, we need to prove that 
\begin{gather}
\frac{\sym{\beta}_\e}{\sqrt{ N_\e |\log \e |    }} \weak   S   \qquad \text{weakly in} \quad \strain \,, \label{conv sym} \\
\frac{\sk{\beta}_\e}{N_\e } \weak   A   \qquad \text{weakly in} \quad \strain \,, \label{conv skew} \\
\lim_{\e \to 0} \frac{1}{N_\e |\log \e|} \int_{\Omega}  W(\beta_\e) \, dx = \int_{\Omega} ( W(S) + \f (\xi) ) \, dx \,.  \label{lim}
\end{gather}
In view of \eqref{stima Q}-\eqref{seq2}, in order to prove  \eqref{conv sym}, \eqref{conv skew} we have to show that
%
\begin{gather}
\frac{\hat{A}_\e }{\sqrt{N_\e \va{\log \e}}}  \weak 0 \quad \text{ in } \quad L^2 (\Omega;\matricidue) \,, \label{th uno} 
\\
\frac{\hat{S}_\e }{\sqrt{N_\e \va{\log \e}}}, \,  \frac{\tilde{K}^{\sigma_\e}_\e }{\sqrt{N_\e \va{\log \e}}}, \, \frac{\tilde{K}_\e^{\eta_\e} }{\sqrt{N_\e \va{\log \e}}} \to 0 \quad \text{ in } \quad L^2 (\Omega;\matricidue) \, \label{th due} \,.
\end{gather}
We have
\begin{equation} \label{th uno prova1}
\begin{aligned}
\int_{\Omega_{\rho_\e} (\mu_\e)}  \frac{ |\hat{A}_\e|^2  }{N_\e |\log \e|   } \, dx  & =  \frac{1}{N_\e |\log \e|} \sum_{i=1}^{M_\e} 
\int_{ B_{r_\e}(x_{\e,i})  \setminus B_{\rho_\e}(x_{\e,i})    }  | \hat{K}^{\xi_{\e,i}}_{\e,i}    |^2 \, dx \\
& \leq \frac{C}{N_\e |\log \e|} \sum_{i=1}^{M_\e} 
\int_{ B_{r_\e}(x_{\e,i})  \setminus B_{\rho_\e}(x_{\e,i})    }  | x - x_{\e,i}    |^{-2} \, dx \\ 
& \leq C \, \frac{M_\e (\log r_\e - \log \rho_\e )}{N_\e |\log \e|} \leq C \, \frac{ \log r_\e - \log \rho_\e }{ |\log \e|} \to 0 \,, 
\end{aligned}
\end{equation}
as $\e \to 0$, where the last inequality follows from \eqref{bound masse}. Moreover, by \eqref{th uno prova1}, \eqref{paper3 conv1}, \eqref{A eps}, \eqref{min self}, and the definition of $\mu_\e^k$ given by Lemma \ref{lemma mis}, we have
\begin{equation} \label{th uno prova2}
\begin{aligned}
\lim_{\e \to 0} \, & \frac{1}{N_\e |\log \e|} \int_\Omega W(\hat{A}_\e) \, dx  = \lim_{\e \to 0} \frac{1}{N_\e |\log \e|} \int_{\Omega \setminus \Omega_{\rho_\e} (\mu_\e)} W(\hat{A}_\e) \, dx \\
& = \lim_{\e \to 0} \frac{1}{N_\e} \sum_{i=1}^{M_\e} \psi(\xi_{\e,i}) (1+o(\e)) =
\lim_{\e \to 0} \frac{1}{N_\e}  \sum_{k=1}^{M}  |\eta_\e^k|(\Omega) \, \psi(\xi_k) (1+o(\e)) \\
& = |\Omega| \sum_{k=1}^M \lambda_k \psi (\xi_k) = \int_\Omega \f (\xi) \, dx \,.  
\end{aligned} 
\end{equation}
From \eqref{coercivity}, \eqref{th uno prova1}, \eqref{th uno prova2} we conclude that 
$\hat{A}_\e / \sqrt{N_\e |\log \e|}$ is bounded in $L^2 (\Omega; \matricidue)$ and its energy is concentrated in the hard core region.
We easily deduce that \eqref{th uno} holds true.

We pass to the proof of \eqref{th due}.  One can readily see that
\[
\int_{\Omega} \frac{  |\tilde{K}_\e^{\sigma_\e}  |^2}{N_\e |\log \e|} \, dx \leq \frac{C}{N_\e |\log \e|} \sum_{i=1}^{M_\e} \frac{1}{r_\e^4} \int_{B_{r_\e} (x_{\e,i})} |x-x_{\e,i}|^2 \, dx = C \, \frac{M_\e}{N_\e |\log \e|} \to 0
\]
as $\e \to 0$. The statement for $\tilde{K}_\e^{\eta_\e}$ can be proved in a similar way.
Finally, since $H_\e \ll N_\e$ by \eqref{bound masse}, we have
\[
\begin{gathered}
\int_{\Omega} \frac{|\hat{S_\e}|^2}{N_\e |\log \e|} \, dx 
\leq C \, \frac{H_\e (\log r_\e - \log \e)}{N_\e |\log \e|} \to 0 
\end{gathered}
\]
which concludes the proof of \eqref{th due} 

We are left to prove \eqref{lim}. 
By the symmetries of the elasticity tensor $\C$ and definition \eqref{rec seq strain}, we have
\begin{equation} \label{paper3 W strain}
\begin{aligned}
\frac{W(\beta_\e) }{N_\e |\log \e|}   =  W \Bigg( 
S + \frac{\hat{S}_\e}{\sqrt{N_\e |\log \e|}} &  - \frac{\tilde{K}^{\sigma_\e}_\e}{\sqrt{N_\e |\log \e|}} + \frac{Q_\e}{\sqrt{N_\e |\log \e|}} + \\
& + \frac{\hat{A}_\e}{\sqrt{N_\e |\log \e|}} - 
\frac{\tilde{K}^{\eta_\e}_\e}{\sqrt{N_\e |\log \e|}} + \frac{R_\e}{\sqrt{N_\e |\log \e|}}   
\Bigg) \,.
\end{aligned}
\end{equation}
 By  \eqref{th due} and \eqref{stima Q}, we get 
\[
\lim_{\e \to 0} \frac{1}{N_\e |\log \e|} \int_{\Omega} W(\beta_\e) \, dx = 
\lim_{\e \to 0}  \int_{\Omega} W \left( S +  \frac{\hat{A}_\e}{\sqrt{N_\e |\log \e|}} \right) \, dx \,. 
\] 
By recalling \eqref{th uno} and \eqref{th uno prova2}, \eqref{th uno}, by H\"older inequality  we deduce \eqref{lim}.

\medskip
\noindent \textbf{Step 2.} The case $\mu = \sum_{l=1}^L \rchi_{\Omega_l} \xi_l \, dx$ and $S \in C^1 (\overline \Omega;\matricidue_{\text{sym}}$).
\medskip

 \noindent In this step we assume that $S \in C^1 (\overline \Omega;\matricidue_{\text{sym}})$ and $A \in \ldueanti$  with $ \mu:= \Curl A$  locally constant, i.e., $\mu= \sum_{l=1}^L \rchi_{\Omega_l} \xi_l \, dx$, with $\xi_l \in \R^2$ and with $\Omega_l \subset \Omega$ that are Lipschitz pairwise disjoint domains such that $|\Omega \setminus \cup_{l=1}^L \Omega_l|=0$. We will construct the recovery sequence by combining the previous step with classical localisation arguments of $\Gamma$-convergence.

Let $S_l := S \zak \Omega_l $, $A_l := A \zak \Omega_l$, $\mu_l := \mu \zak \Omega_l = \xi_l \, dx$. Denote by $(\mu_{l,\e}, \beta_{l,\e})$ the recovery sequence for $(\mu_l,S_l,A_l)$ given by Step 1. We can now define $\mu_\e \in \misure$ and ${\beta}_\e \colon \Omega \to \matricidue$ as
\[
{\beta}_\e := \sum_{l=1}^L \rchi_{\Omega_l} \, \beta_{l,\e} \,, \quad \mu_\e :=\sum_{l=1}^L \mu_{l,\e} \,.
\]
By construction $\mu_\e \in \AD_\e (\Omega)$ and ${\beta}_\e$ satisfies the circulation condition on every $\de B_\e (x_\e)$, with $x_\e \in \spt \mu_\e$. Also notice that on each set $\Om_l$ belonging to the partition of $\Om$, we have  
\[
\Curl {\beta}_\e \zak \Om_l (\mu_{\e})=0\,.
\]
However $\Curl \beta_\e$ could concentrate on the intersection region between two elements of the partition $\{\Om_l\}_{l=1}^L$. To overcome this problem, it is sufficient to notice that by construction 
\[
\begin{aligned}
\nor{  \frac{  \Curl \beta_\e  \zak \Omega_\e (\mu_\e)      }{ \sqrt{N_\e |\log \e|}  }          }_{H^{-1}(\Omega; \R^2)} & \leq 
\sum_{l=1}^L \nor{  \frac{ \beta_{l,\e} - \sqrt{N_\e |\log \e|} S - N_\e A    }{\sqrt{N_\e |\log \e|}}  \cdot t}_{H^{-1/2} (\de \Omega_l;\R^2)}  \\
& = \sum_{l=1}^L \nor{  \frac{ Q_{l,\e} +  R_{l,\e}}{\sqrt{N_\e |\log \e|}}  \cdot t}_{H^{-1/2} (\de \Omega_l;\R^2)} \,,
\end{aligned}
\]
where $Q_{l,\e}$, $R_{l,\e}$ are defined according to \eqref{laplace 1}, with $\Omega$ replaced by $\Omega_l$. 
Therefore by \eqref{de om}, 
\[
\frac{  \Curl \beta_\e  \zak \Omega_\e (\mu_\e)      }{ \sqrt{N_\e |\log \e|}  }  \to 0 \quad \text{strongly in } \quad H^{-1} (\Omega; \R^2) \,.
\]
Hence we can add a vanishing perturbation to $\beta_\e$ (on the scale $\sqrt{N_\e |\log \e|}$), in order to obtain the desired recovery sequence in $\AS_\e (\mu_\e)$.

\medskip
\noindent \textbf{Step 3.} \textit{The general case.} 
\medskip

\noindent Let $(\mu,S,A)$ be in the domain of the $\Gamma$-limit $\mathcal{F}$. 
In view of Step 2 and by standard density arguments of $\Gamma$-convergence, it is sufficient to find sequences
$(\mu_n, S_n,A_n)$ such that $\mu_n$   is locally constant as in Step 2, 
\begin{equation}
\begin{gathered} \label{paper3 succ densa}
S_n  \in C^1 (\overline \Omega; \matricidue_{\rm sym}) \,, \,\, 
A_n \in L^2 (\Om;\matricidue_{\rm skew}) \,, \,\, \text{ with } \,\, \Curl A_n = \mu_n,
\end{gathered}	
\end{equation}
and such that 
\begin{equation} \label{density}
S_n \to S, \,\,\, A_n\to A \text{ in } \, L^2 (\Omega;\matricidue)\,, \quad \mu_n \weakstar \mu \, \text{ in } \, \misure \,, \quad  |\mu_n| (\Omega) \to |\mu| (\Omega),	
\end{equation}
where $S$ and $A$ are the symmetric and antisymmetric part of $\beta$, respectively. 
In fact , we have to  show that \eqref{density} implies
\begin{equation} \label{paper3 gam den}
\lim_{n \to \infty} \mathcal{F}(\mu_n,\beta_n) = \mathcal{F} (\mu,S,A) \,. 
\end{equation}
Since $S_n \to S$ strongly in $L^2 (\Omega;\matricidue)$, then 
\[
\lim_{n \to \infty} \int_{\Om} W(S_n) \, dx = \int_{\Om} W(S) \, dx \,. 
\]
Also, $|\mu_n|(\Om) \to |\mu|(\Om)$ implies 
\[
\lim_{n \to \infty} \int_{\Om} \f \left( \frac{d \mu_n}{d|\mu_n|}   \right) \, d|\mu_n| = \int_{\Om} \f \left( \frac{d \mu}{d|\mu|}   \right) \, d|\mu| \,,
\]
by Reshetnyak's Theorem (\cite[Theorem 2.39]{afp})), so that \eqref{paper3 gam den} is proved. 

Let us then proceed to the construction of the sequences $S_n$,  $A_n$  and $\mu_n$ satisfying properties \eqref{paper3 succ densa}-\eqref{density}. 
Clearly, we can approximate $S$ in $\lduesimm$ with a sequence 
$S_n \in C^1 (\overline \Omega; \matricidue_{\rm sym})$. 
Then, by Remark \eqref{umu}, writing $A$ as  in \eqref{au} we have that $u$ is in $BV(\Om)\cap L^2(\Om)$.  Therefore, by  standard density results in $BV$ we can find a 
sequence of piecewise affine functions  $u_n$ with 
$$
u_n\to u \text{ in } L^2(\Om), \qquad Du_n \weakstar Du = \mu, \qquad |D u_n|(\Om) \to |Du|(\Om) = |\mu|(\Om).
$$
Setting $\mu_n:= D u_n$ and $A_n$ as in \eqref{au} with $u$ replaced by $u_n$, it is readily seen that $\mu_n$ is piecewise constant, and that \eqref{paper3 succ densa} and \eqref{density} holds true, and this concludes the proof of the $\Gamma$-limsup inequality. 
\end{proof}

\begin{remark}\label{stimeg}
Recalling \eqref{de om} and inspecting the density arguments in Step 3 above,  we notice that we can provide a recovery sequence $\beta_\e$ for the limit strain $\beta=S+A$ such that 
\begin{equation}\label{stimameglio0}
 \frac{\beta_\e}{N_\e} \cdot t \to A\cdot t \text{ in } H^{-1/2}(\de\Om;\R^2)\cap L^1(\de\Om;\R^2).
\end{equation}
\end{remark}

\section{Relaxed Dirichlet-type boundary conditions} \label{paper3 boundary}

The aim of this section is to add a Dirichlet type boundary condition to the $\Gamma$-convergence statement of Theorem \ref{thm:gamma}. Fix a boundary condition 
\begin{equation} \label{paper3 bc1}
 g_A \in \ldueanti : \, \Curl g_A \in  H^{-1} (\Omega;\R^2) \cap \misure.
\end{equation}


The rescaled energy functionals $\mathcal{F}^{g_A}_\e \colon \misure \times \strain \to \R$, taking into account the boundary conditions,  are defined by
\begin{equation}\label{energyprimo}
\mathcal{F}^{g_A}_\e (\mu, \beta) :=
\frac{1}{N_\e |\log \e |} \, E_\e (\mu, \beta) + \int_{\de \Om} \f \Big( \Big(g_A-\frac{\beta}{N_\e}\Big) \cdot t \Big) \, ds 
\end{equation}
if $\mu \in \AD_\e (\Omega) \,,\, \beta \in \AS_\e (\mu)$, and
$ + \infty$ otherwise,   
while the candidate $\Gamma$-limit is the functional 
\begin{equation}\label{gbd}
\mathcal{F}^{g_A} \colon ( \duale \cap \misure ) \times \lduesimm  \times \ldueanti \to \R \,,
\end{equation}
with
\begin{equation} \label{limite primo}
\mathcal{F}^{g_A} (\mu, S,A):= 
 \displaystyle \int_\Omega W(S) \, dx + \int_\Omega \f \left( \displaystyle \frac{d \mu}{d |\mu|} \right) \, d |\mu|   + \int_{\de \Om} \f ( (g_A-A) \cdot t ) \, ds \,,  
\end{equation}
if $\Curl A = \mu$ and $\mathcal{F}^{g_A} (\mu, S,A):= \infty$ otherwise. Here $ds$ coincides with $\mathcal{H}^1 \zak \de \Om$. 
The boundary term appearing in the definition of $\mathcal{F}^{g_A}_\e$ and  $\mathcal{F}^{g_A}$ are intended in the sense of traces of $BV$ functions (see \cite{afp}). Indeed, since $A$ and $g_A$ are antisymmetric, there exist $u, a \in L^2(\Om)$ such that
\[
A=\left( \begin{matrix}
 0 &  u \\
 -u  &  0 \\	
 \end{matrix}
\right) \,, \quad g_A=\left( \begin{matrix}
 0 &  a \\
 -a  &  0 \\	
 \end{matrix}
\right) \,,
\]
Notice that $\Curl A = D u$ and $\Curl g_A = D a$ in the sense of distributions. Therefore, as already observed in Remark \ref{umu}, conditions $\Curl A, \Curl g_A \in  \misure$ imply that $a, u \in BV(\Om)$. Hence $a$ and $u$ admit traces on $\de \Om$ that belong to $L^1(\de\Om;\R^2)$. By noting that 
\[
\int_{\de \Om} \f ((g_A-A)\cdot t) \, ds = \int_{\de \Om} \f ((u-a) \nu) \, ds \,,
\]
where $\nu$ is the inner normal to $\Om$, we conclude that the definition of $\mathcal{F}^{g_A}$ is well-posed, as well as the definition of  $\mathcal{F}^{g_A}_\e$.

We are now ready to state the $\Gamma$-convergence result with boundary conditions.

\begin{theorem}\label{gamma conv bc} The following $\Gamma$-convergence statement holds with respect to the convergence of Definition \ref{def:conv}.
\begin{enumerate}[(i)]
\item \textnormal{(Compactness)} Let $\e_n \to 0$ and assume that $(\mu_n,\beta_n) \in \misure \times \strain$ is such that $\sup_n \mathcal{F}^{g_{A}}_{\e_n} (\mu_n,\beta_n ) \leq E$, for some positive constant $E$. Then there exists $(\mu,S,A) \in ( \duale \cap \misure ) \times \lduesimm \times \ldueanti$ such that 
$(\mu_n,\beta_n)$ converges to $(\mu,S,A)$ in the sense of Definition \ref{def:conv}. Moreover $\mu \in \duale$ and $\Curl A = \mu$. 

\item \textnormal{(}$\Gamma$\textnormal{-convergence)}
The energy functionals $\mathcal{F}^{g_A}_\e$ defined in \eqref{energyprimo} $\Gamma$-converge with respect to the convergence of Definition \ref{def:conv} to the functional $\mathcal{F}^{g_A}$ defined in \eqref{limite primo}. Specifically, for every
\[ 
 (\mu,S,A) \in (\misure \, \cap \, \duale) \times \lduesimm \times \ldueanti
\] 
such that $\Curl A= \mu$, we have:
 \begin{itemize}
 \item   \textnormal{(}$\Gamma$\textnormal{-liminf inequality)}
for every sequence $(\mu_\e,\beta_\e) \in \misure \times \strain$ converging to $(\mu,S,A)$ in the sense of Definition \ref{def:conv}, we have
\[
\mathcal{F}^{g_A} (\mu,S,A) \leq \liminf_{\e \to 0} \mathcal{F}^{g_A}_\e (\mu_\e,\beta_\e) \,.
\]
\item 	\textnormal{(}$\Gamma$\textnormal{-limsup inequality)}
there exists a recovery sequence $(\mu_\e,\beta_\e)\in
\misure \times \strain$ such that $(\mu_\e,\beta_\e)$ converges to $(\mu,S,A)$ in the sense of \mbox{Definition \ref{def:conv},} and
\[
\limsup_{\e \to 0} \mathcal{F}^{g_A}_\e (\mu_\e,\beta_\e)  \leq \mathcal{F}^{g_A} (\mu,S,A)  \,.
\]
 \end{itemize}
\end{enumerate}
\end{theorem}

The compactness statement readily follows from the compactness of Theorem  \ref{thm:gamma}, since $\mathcal{F}^{g_A}_\e (\mu,\beta) \ge \mathcal{F}_\e (\mu,\beta)$. Let us proceed with the proof of the $\Gamma$-convergence result.

\begin{proof}[Proof of $\Gamma$-$\limsup$ inequality of Theorem \ref{gamma conv bc}.] 
Let $(\mu,S,A)$ be given in the domain of the $\Gamma$-limit $\mathcal{F}^{g_A}$. We will construct a recovery sequence in two steps, relying on Theorem \ref{thm:gamma}. 

\medskip
\noindent \textbf{Step 1.} Approximation of the boundary values.
\medskip

\noindent For $\delta>0$ fixed, set 
$\omega_\delta := \{ x \in \Omega \, \colon \dist(x,\de \Om) > \delta \}$, so that $\omega_\delta \subset \subset \Om$, and assume without loss of generality that $w_\delta$ is Lipschitz. 
Define $S_\delta \in \lduesimm$ and $A_\delta \in \ldueanti$ as
\begin{equation} \label{strain modificato}
A_\delta := \begin{cases} 
A   &   \text{ in } \, \omega_\delta\,, \\ 	
g_A   &  \text{ in } \, \Omega \setminus \omega_\delta \,, \\
 \end{cases} \qquad
S_\delta := \begin{cases} 
S   &   \text{ in } \, \omega_\delta \,, \\ 	
0   &  \text{ in } \, \Omega \setminus \omega_\delta \,. \\
 \end{cases} 
\end{equation}
Further, let $\mu_\delta \in \misure$ be such that
\begin{equation} 
\mu_\delta := \mu \zak \omega_\delta  + \Curl g_A  \zak (\Om \setminus \omega_\delta) + (g_A- A) \cdot t \,\, \mathcal{H}^1 \zak \de \omega_\delta \, . \label{misura modificata}
\end{equation}
Notice that 
\begin{equation} \label{curl dato al bordo}
\Curl A_\delta = \mu_\delta \quad \text{ and } \quad \mu_\delta \in H^{-1} (\Om;\R^2) \,,
\end{equation}
therefore $(\mu_\delta,S_\delta,A_\delta)$ belongs to the domain of the functional $\mathcal{F}$. 
Also note that
\begin{equation}  \label{densita dato}
\begin{gathered}	
S_\delta \to S  \,, \, A_\delta \to A  \, \text{ in } \, L^2 (\Omega; \matricidue) \,, 
\\
\mu_\delta \weakstar \mu \, \text{ in } \, \misure \,, \, |\mu_\delta|(\Om) \to |\mu|(\Om) + \int_{\partial \Om} |(g_A -A)\cdot t| \, ds \,, 
\end{gathered}
\end{equation}
as $\delta \to 0$. Therefore, by Reshetnyak's Theorem (see \cite[Theorem 2.39]{afp}), we have
\begin{equation} \label{densita energia dato}
\lim_{\delta \to 0}  \mathcal{F}(\mu_\delta, S_\delta, A_\delta) = \mathcal{F}^{g_A} (\mu,S,A) \,.	
\end{equation}
It will now be sufficient to construct dislocation measures $ \mu_{\delta,\e}$  and strains $\beta_{\delta,\e}$ 
such that $(\mu_{\delta,\e}, \beta_{\delta,\e})$ converges to $(\mu_\delta,S_\delta,A_\delta)$ in the sense of Definition \ref{def:conv} and that
\begin{equation} \label{thesis dato}
\lim_{\e \to 0} \mathcal{F}^{g_A}_{\e} (\mu_{\delta,\e}, \beta_{\delta,\e} ) =
\mathcal{F} (\mu_{\delta}, S_{\delta},A_\delta) \,.
\end{equation}
Indeed, by taking a diagonal sequence $(\mu_{\delta_{\e},\e}, \beta_{\delta_{\e},\e})$ and using \eqref{densita dato}, \eqref{densita energia dato}, the thesis will follow.

\medskip
\noindent \textbf{Step 2.} Recovery sequence for strains satisfying the boundary condition. 
\medskip

\noindent Let us now proceed to construct the sequence $(\mu^{g_\e}_{\delta,\e}, \beta^{g_\e}_{\delta,\e})$ as stated in the previous step.

From Theorem \ref{thm:gamma}, there exist a sequence  $(\mu_{\delta,\e}, \beta_{\delta,\e})$ 
converging to $(\mu, S, A)$ in the sense of Definition \ref{def:conv}, 
and such that 
\begin{equation} \label{conv energia}
\lim_{\e \to 0} \mathcal{F}_{\e} (\mu_{\delta,\e}, \beta_{\delta,\e}) = \mathcal{F}(\mu,S, A) \,.
\end{equation}
Moreover (see Remark \ref{stimeg}), we can assume that
$\beta_\e$ satisfies \eqref{stimameglio0}, from which it easily follows \eqref{conv energia}.

\end{proof}

\begin{proof}[Proof of $\Gamma$-$\liminf$ inequality of Theorem \ref{gamma conv bc}.]
Let $(\mu,S,A)$ be in the domain of the $\Gamma$-limit $\mathcal{F}^{g_A}$. Assume that $(\mu_\e,\beta_\e)$
converges to $(\mu,S,A)$ in the sense of Definition \ref{def:conv}. By combining an extension argument with the $\Gamma$-$\liminf$ inequality in  Theorem \ref{thm:gamma} we will show that
\begin{equation} \label{thesis liminf}
\mathcal{F}^{g_A} (\mu,S,A) \leq \liminf_{\e \to 0} \mathcal{F}^{g_A}_\e (\mu_\e,\beta_\e) \,.
\end{equation}

Fix $\delta>0$ and define $U_\delta := \{ x \in \R^2 \, \colon \, \dist (x, \Omega)<\delta \}$. By standard reflexion arguments one can extend $g_A$ to 
 $\tilde{g}_A \in L^2(U_\delta ; \matricidue_{\rm skew})$, in such a way that
 $\tilde{\mu}_A:= \Curl \tilde{g}_A$ is a  measure on 
$U_\delta$ satisfying  $|\tilde{\mu}_A|(\de \Om) = 0$. 
Consider now the functions $\tilde \beta_\e$ defined as in \eqref{defbt} (with $\e_n$ replaced by $\e$), and set
$$
\hat \beta_\e:=
\begin{cases}
	\tilde \beta_\e &  \text{ in } \, \Om \,, \\
		N_\e \tilde{g}_A  &  \text{ in } \, U_\delta \setminus \Om \,,
\end{cases}
\qquad
\hat \beta:=
\begin{cases}
	A &  \text{ in } \, \Om \,, \\
		\tilde{g}_A  &  \text{ in } \, U_\delta \setminus \Om \,.
\end{cases}
$$
By   construction we have $\frac{\hat \beta_\e}{N_\e}\weak \hat \beta$ in $L^1(U_\delta)$, 
so that
$$
\hat \mu_\e := \frac{\Curl \hat \beta_\e}{N_\e} \weakstar \mu + ((g_A-A) \cdot t ) \mathcal H^1 \zak \partial \Om + \Curl \tilde g_A \zak {(U_\delta\setminus \Om)}
$$
Recalling \eqref{defbt2}, \eqref{paper3 liminf1} and \eqref{paper3 liminf2}, we conclude
\begin{multline}
\liminf_{\e \to 0} \mathcal{F}^{g_A}_\e (\mu_\e,\beta_\e)\ge 
\liminf_{\e \to 0}  \frac{1}{\sqrt{N_\e |\log\e|} } \int_{\Om} W(\sym{\beta}_\e) \, dx   
\\
+
\liminf_{\e \to 0} \frac{1}{N_\e}  \int_\Omega \f \left(\frac{d \mu_\e}{d | \mu_\e| }\right) d | \mu_\e|   + \int_{\de \Om} \f \Big( \Big(g_A-\frac{\beta_\e}{N_\e}\Big) \cdot t \Big) \, ds    
\\
\ge 
\int_{\Om} W(S) \, dx  +
\liminf_{\e \to 0}   \int_{U_\delta} \f \left(\frac{d \hat \mu_\e}{d |\hat \mu_\e|}  \right) d |\hat \mu_\e| 
- \int_{U_\delta\setminus \Om} \f \left(\frac{d \Curl \tilde g_A}{d |\Curl \tilde g_A| } \right) d |\Curl \tilde g_A|\\
\ge
\int_\Omega W(S) \, dx + \int_\Omega \f \left( \displaystyle \frac{d \mu}{d |\mu|} \right) \, d |\mu|   + \int_{\de \Om} \f ( (g_A-A) \cdot t ) \, ds =
\mathcal{F}^{g_A}_\e (\mu,S,A).
 \end{multline}
\end{proof}

\section{Linearised polycrystals as minimisers of the $\Gamma$-limit} \label{paper3 poly}

Let $\Om \subset \R^2$ be a bounded domain with Lipschitz continuous boundary. 
Let $k \in \N$ be fixed and let $\{U_i\}_{i=1}^k$ be a Caccioppoli partition of $\Omega$ (see \cite[Section 4.4]{afp}). Moreover fix $m_1,\dots,m_k \in \R_+$ with $m_i<m_{i+1}$, and define the piecewise constant function $a \in BV(\Om)$ as
\begin{equation} \label{paper3 def a}
a:=\sum_{i=1}^k m_i \rchi_{U_i} \,.
\end{equation}
In particular, \eqref{paper3 def a} implies that $a \in L^{\infty} (\Om)$ and $D a \in \misure$.   
We can now define the piecewise constant boundary condition $g_A \in L^{\infty} (\Om;\matricidue_{\rm skew})$ as
\begin{equation} \label{paper3 ga}
g_A := \left(
\begin{matrix}
0    &   a \\
-a    &    0	
\end{matrix}
\right).
\end{equation}
Notice that $g_A \in L^2 (\Om;\matricidue_{\rm skew})$ and $\Curl g_A = D a$, therefore $\Curl g_A \in \duale \cap \misure$. In this way $g_A$ is an admissible boundary condition for $\mathcal{F}^{g_A}$, as required in \eqref{paper3 bc1}.  

We wish to minimise the $\Gamma$-limit \eqref{limite primo} with boundary condition $g_A$ prescribed by \eqref{paper3 def a}-\eqref{paper3 ga}.  
Since the elastic energy and plastic energy are decoupled in $\mathcal{F}^{g_A}$, and there is no boundary condition fixed on the elastic part of the strain $S$, we have 
\[
\inf \mathcal{F}^{g_A}(\Curl A,S,A)= \inf \mathcal{F}^{g_A}(\Curl A,0,A) \,.
\]
Therefore it is sufficient to study
\begin{equation} \label{prob min bc}
\begin{aligned}
\inf \Bigg\{  \int_{\Omega} \f (\Curl A)  + \int_{\de \Omega}  \f ((g_A-A) & \cdot t) \, ds \, \colon \, A \in \ldueanti, \\
& \Curl A \in H^{-1} (\Om ;\R^2) \cap \mathcal{M} (\Om;\R^2)   \Bigg\} \,,
\end{aligned}
\end{equation}
where $t$ is the unit tangent to $\de \Om$ defined as the $\pi/2$ counter-clockwise rotation of the outer normal $\nu$ to $\Om$, $\f \colon \R^2 \to [0,\infty)$ is the density defined in \eqref{self}, and
\begin{equation} \label{anisotropic perimeter}
\int_{\Omega} \f (\mu) :=  \int_{\Omega} \f \left( \frac{d \mu}{ d |\mu|  } \right)  \, d |\mu|
\end{equation}
is the anisotropic $\f$-total variation for a measure $\mu \in \misure$. Note that \eqref{anisotropic perimeter} is well-posed, since $\f$ satisfies the properties given in Proposition \ref{paper3 properties density}. 

For $A \in L^2(\Om;\matricidue_{\rm skew})$, we have that 
\begin{equation} \label{paper3 A}
A=  \left(
\begin{matrix}
0    &   u \\
-u    &    0	
\end{matrix} 
\right) \,,
\end{equation}
for some $u \in L^2(\Om)$. Moreover $\Curl A = D u$, therefore condition $\Curl A \in \misure$ implies $u \in BV(\Om)$. Also notice that 
\[
\int_{\de \Om} \f ((g_A-A)\cdot t) \,ds = 
\int_{\de \Om} \f ((u-a) \nu) \,ds\,, 
\]
where $a$ is the piecewise constant function \eqref{paper3 def a}.  We claim that \eqref{prob min bc} is equivalent to the following minimisation problem 
\begin{equation} \label{conj}
\inf \left\{ \int_{\Omega} \f (D u)  + \int_{\de \Omega}  \f ((u-a)\nu) \, ds \, \colon \, u \in BV (\Omega) \right\} \,.
\end{equation}
Indeed, we already showed that if $A$ is a competitor for \eqref{prob min bc}, then the function $u$, given by \eqref{paper3 A}, belongs to $BV(\Om)$, and it is a competitor for \eqref{conj}. Conversely, assume that $u \in BV(\Om)$ and define $A$ through \eqref{paper3 A}. Since $u \in BV(\Om)$, then $\Curl A = Du \in \misure$. Moreover, recall that the immersion $BV (\Om) \hookrightarrow L^2(\Om)$ is continuous, therefore $u \in L^2(\Om)$, which implies $A \in L^2(\Om;\matricidue)$, so that $\Curl A \in \duale$. This shows that \eqref{prob min bc} and \eqref{conj} are equivalent.

The main result of this section is that, given the piecewise constant boundary condition $a$ defined in \eqref{paper3 def a}, there exists a piecewise constant minimiser $\tilde{u}$ to \eqref{conj}. In our model the function $\tilde{u}$ corresponds to a linearised polycrystal.

\begin{theorem} \label{prop:loc const}
There exists a locally constant minimiser $\tilde{u} \in BV(\Om)$ to \eqref{conj}, i.e.,
\[
\tilde{u}= \sum_{i=1}^k m_i \rchi_{\Om_i},
\]
where $\{\Om_i\}_{i=1}^k$ is a Caccioppoli partition of $\Omega$, and the values $m_i$ are the ones of \eqref{paper3 def a}.  
\end{theorem}

The proof of this theorem relies on the anisotropic coarea formula. For the readers convenience we briefly recall it here. 
For $E \subset \Omega$ of finite perimeter, the anisotropic $\f$-perimeter of $E$ in $\Omega$ is defined as 
\[
\Per_\f (E,\Om) := \int_{\Omega} \f (D \rchi_E) \,. 
\]
Since $\f$ is convex and positively 1-homogenous, the anisotropic coarea formula holds true for every $u \in BV(\Om)$:
\begin{equation} \label{coarea}
	\int_{\Omega} \f (D u) = \int_{-\infty}^{\infty} \Per_{\f} (E_t,\Om) \, dt \,,
\end{equation}
where $E_t$ is the level set $E_t:= \{ x \in \Omega \, \colon \, u(x)>t \}$, defined for every $t \in \R$.

\begin{proof}[Proof of Theorem \ref{prop:loc const}.] 
\hfill

\noindent \textbf{Step 1.} Equivalent minimisation problem.

\noindent We start by rewriting \eqref{conj} as a boundary value problem in $BV$. Let $\Omega':=\{ x \in \R^2 \, \colon \, \dist (x, \Om)< 1 \}$, so that $\Om \subset \subset \Om'$.  Consider a piecewise constant extension $\tilde{a} \in BV(\Om')$ of the function $a \in BV(\Om)$ defined in \eqref{paper3 def a}, that is,
\[
\tilde{a} = \sum_{i=1}^{k} m_i \, \rchi_{U_i'}\,,
\] 
where $\{ U_i'\}_{i=1}^k$ is a Caccioppoli partition of $\Om'$, agreeing with $\{ U_i\}_{i=1}^k$ on $\Om$. 
This is possible since the extension can be chosen such that $|D \tilde{a}|(\de \Om)=0$, that is, we are not creating any jump on $\de \Om$. 
Consider the new minimisation problem
\begin{equation} \label{prob ausiliario}
I:= \inf \left\{ \int_{\Om '} \f (D u) \, \colon \, u \in BV(\Om'), \,\, u= \tilde{a} \text{ a.e. in } \Om' \setminus \Om \right\} \,. 
\end{equation}
Finding a solution to \eqref{prob ausiliario} is equivalent to finding a solution to \eqref{conj}. 
Indeed, if $u \in BV(\Om')$ is such that $u= \tilde{a}$ in $\Om' \setminus \Om$ then  
\begin{equation} \label{paper3 app5}
Du = Du \, \zak \Om  + (u^\Om - a^\Om) \, \nu \, \mathcal{H}^1 \zak \de \Om + D\tilde{a} \, \zak ( \Om' \setminus \Om ) \,,
\end{equation}
where $u^\Om, a^\Om \in L^1(\de \Om)$ are the traces of $u$ and $a$ on $\de \Om$. Notice that we can use $a^\Om$ in \eqref{paper3 app5} because the extension $\tilde{a}$ is such that $|D\tilde{a}|(\de \Om)=0$, hence we have $\tilde{a}^+_{\de \Om} = \tilde{a}^-_{\de \Om}=a^{\Om}$ $\mathcal{H}^{n-1}$-a.e. in $\de \Om$.

\smallskip
\noindent \textbf{Step 2.} Existence of a minimiser for \eqref{prob ausiliario}. 
\smallskip

\noindent Let $u_j \in BV(\Om')$ be a minimising sequence for \eqref{prob ausiliario}, that is $u_j = \tilde{a}$ a.e. on $\Om' \setminus \Om$ and 
\begin{equation} \label{paper3 app6}
\lim_{j \to \infty} \int_{\Om'} \f (D u_j) = I \,.
\end{equation}
By standard truncation arguments we can assume that 
$
\|u_j\|_\infty \le \max_i |m_i|.
$
In particular, from \eqref{paper3 app6}, we deduce that $\sup_j \nor{u_j}_{BV(\Om')}<\infty$. By compactness in $BV$, there exists $\tilde{u} \in BV(\Om')$ such that, up to subsequences, $u_j \to \tilde{u}$ in $L^1(\Om')$ and $Du_j \weakstar D\tilde{u}$ weakly in $\mathcal{M}(\Om';\R^2)$. Since $u_j = \tilde{a}$ a.e. on $\Om' \setminus \Om$, the strong convergence in $L^1$ implies that (up to subsequences) $u_j \to \tilde{u}$ a.e. in $\Om'$, so that $\tilde{u} = \tilde{a}$ a.e. in $\Om' \setminus \Om$. From Reshetnyak's lower semicontinuity Theorem we conclude that
\[
\int_{\Om'} \f(D\tilde{u}) \leq \liminf_{j \to \infty} \int_{\Om'} \f(D u_j) = I \,,
\]
so that $\tilde{u}$ is a minimiser for \eqref{prob ausiliario}.

\smallskip
\noindent \textbf{Step 3.} Existence of a piecewise constant minimiser for \eqref{conj}.
\smallskip 

\noindent Let $u$ be a minimiser for \eqref{prob ausiliario}. By a standard truncation argument we can assume that $m_1 \leq u \leq m_k$ a.e. on $\Omega'$. Formula \eqref{coarea} then reads
\begin{equation} \label{conj1}
\int_{\Omega '} \f (Du) = \sum_{i=1}^{k-1} \int_{m_i}^{m_{i+1}} \Per_\f (E_t , \Om') \, dt  \,,
\end{equation}
where $E_t:=\{ x \in \Om' \colon u(x)>t\}$ for $t \in \R$. 
By the mean value theorem, for every $i=1,\dots,k-1$, there exists   $t_i \in (m_i,m_{i+1})$ such that
\begin{equation} \label{conj2}
\int_{m_i}^{m_{i+1}} \Per_\f (E_t , \Om') \, dt \geq (m_{i+1}-m_i) \Per_\f (E_{t_i},\Om')\,.
\end{equation}

We define the piecewise constant function
\[
\tilde{u}(x):= 
m_i      \qquad  \text{ if }  	x \in E_{t_{i-1}} \smallsetminus E_{m_i}\, , 
\]
for $i=1,\dots,k$, where we have set $E_{t_0}:=\Om'$ and we notice   that $E_{m_k}= \emptyset$ set theoretically.
Since the sets $E_t$ have finite perimeter in $\Om'$, we have that $\tilde{u} \in BV(\Omega')$. Moreover, by construction, $\tilde{u} = \tilde{a}$ on $\Om' \setminus \Om$, so that $\tilde{u}$ is a piecewise constant competitor for \eqref{prob ausiliario}. It is immediate to compute that
\[
D \tilde{u} = \sum_{i=1}^{k-1} (m_{i+1}-m_i) \, \nu_{E_{t_i}} \, \mathcal{H}^1 \zak \de^* E_{t_i} \,,
\] 
so that
\begin{equation} \label{conj3}
\begin{aligned}
\int_{\Omega'} \f (D \tilde{u}) & = \sum_{i=1}^{k-1} (m_{i+1}-m_i) \int_{\de^* E_{t_i}} \f(\nu_{E_{t_i}}) \, d \mathcal{H}^1 \\
& = \sum_{i=1}^{k-1} (m_{i+1}-m_i)  \Per_\f (E_{t_i},\Om') \,.
\end{aligned}
\end{equation}
By minimality of $u$ and \eqref{conj1}-\eqref{conj3} we conclude that $\tilde{u}$ is a locally constant minimiser for \eqref{prob ausiliario}. Hence $\left. \tilde{u} \right|_{\Om}$ is a locally constant minimiser for \eqref{conj}.
\end{proof}

\section{Conclusions and perspectives} \label{paper3 conclusions}

The aim of this paper is to describe polycrystalline structures from a variational point of view. Grain boundaries and the corresponding grain orientations are not introduced as internal variables of the energy, but they spontaneously arise as a result of energy minimisation, under suitable boundary conditions. 

We work under the hypothesis of linear planar elasticity as in \cite{glp}, with the reference configuration $\Om \subset \R^2$ representing a section of an infinite cylindrical crystal. The elastic energy functional depends on the lattice spacing $\e$ of the crystal and we allow $N_\e$ edge dislocations in the reference configuration, with $N_\e \to \infty$ as $\e \to 0$. Each dislocation contributes by a factor $|\log \e|$ to the elastic energy, so that the natural rescaling for the energy functional is $N_\e |\log \e|$. We work in the energy regime
\[
\frac{1}{\e} \gg N_\e \gg |\log \e| \,,
\]
which accounts for grain boundaries that are mutually rotated by an infinitesimal angle $\theta \approx 0$. 

After rescaling the elastic energy of such system of dislocations and sending the lattice spacing $\e$ to zero, 
in Theorem \ref{thm:gamma} we derive by $\Gamma$-convergence a macroscopic energy functional of the form
\[
\mathcal{F}(\mu, S,A) = \int_\Om \C S : S \,dx + \int_{\Om} \f \left( \frac{d \mu}{d |\mu|} \right) \, d|\mu| \,, 
\]
where $\C$ is the linear elasticity tensor and $\f$ is a positively $1$-homogeneous density function, defined through a suitable cell-problem. The elastic energy is computed on $S$, that represents the elastic part of the macroscopic strain. The plastic energy depends only on the dislocation measure $\mu$, which is coupled to the plastic part $A$ of the macroscopic strain through the relation $\mu=\Curl A$. As a consequence, $\mu$ is a curl-free vector Radon measure. The contributions of elastic energy and plastic energy are decoupled in the $\Gamma$-limit $\mathcal{F}$, due to the fact that $S$ and $A$ live on different scales: $\sqrt{N_\e |\log \e|}$ and $N_\e$, respectively. 

Indeed this is the main difference with the energy regime $N_\e \approx |\log \e|$ studied in \cite{glp}, where $S$ and $A$ live on the same scale $|\log \e|$. In \cite{glp} the authors derive a macroscopic energy that has the same structure as $\mathcal{F}$, but in which the contributions of elastic energy and plastic energy are coupled by the relation $\mu = \Curl \beta$, where $\beta = S+A$ represents the whole macroscopic strain.  

Once the $\Gamma$-limit $\mathcal{F}$ is obtained, we impose a piecewise constant Dirichlet boundary condition on $A$, and minimise $\mathcal{F}$ under such constraint. In Theorem \ref{prop:loc const} we prove that $\mathcal{F}$ admits piecewise constant minimisers, of the form
\[
\hat{A}= \sum_{i=1}^k  A_i \, \rchi_{\Om_i} \,,
\]
where the $A_i$'s are antisymmetric matrices and $\{ \Om_i\}$ is a Caccioppoli partition of $\Om$. We interpret $\hat{A}$ as a linearised polycrystal, with $\Om_i$ representing a single grain having  orientation $A_i$. This interpretation is motivated by the fact that antisymmetric matrices can be considered as infinitesimal rotations. The (linear) energy corresponding to $\hat{A}$ can be seen as a linearised version of the Read-Shockley formula for small angle tilt grain boundaries, i.e.,
\begin{equation} \label{pers shock}
E= E_0 \, \theta (1+|\log \theta|)\,,
\end{equation}
where $E_0 >0$ is a constant depending only on the material and $\theta $ is the angle formed by two grains. 
Indeed, the Read-Shockley formula is obtained in \cite{shockley} by computing the elastic energy for an evenly spaced array of $1/\e$ dislocations at the grain boundaries. Our energy regime accounts only for $N_\e \ll 1/\e$ dislocations, therefore we do not have enough dislocations to create true rotations between grains. Nevertheless we still observe polycrystalline structures, but the rotation angles between grains are infinitesimal.

Recently Lauteri and Luckhaus \cite{lauteri} proved some compactness properties and energy bounds in agreement  with the Read-Shockley formula.  It would be desirable to understand if our $\Gamma$-limit can be deduced from their model as the angle $\theta$ between grains tends to zero.  
Moreover, it would be interesting to push our $\Gamma$-convergence analysis to energy regimes of order $\frac{|\log\e|}{\e}$, corresponding to $N_\e \approx \frac{1}{\e}$. In this regime true  rotations should emerge and the Read-Shockley formula could be possibly derived by $\Gamma$-convergence. At present, our technical assumption on well separation between dislocations is not compatible with such an energy regime.

Another natural question is whether the minimiser $\hat{A}$ is unique, or at least if all the minimisers are piece-wise constant. 
We suspect that, by enforcing piece-wise constant boundary conditions, generically all minimisers are piece-wise constant.

A further problem is to deduce our $\Gamma$-limit $\mathcal{F}$ by starting from a nonlinear energy computed on small deformations $v=x + \e u$, in the energy regime $N_\e \gg |\log \e|$. A similar analysis was already performed in \cite{msz} (see also \cite{sz, ginster}), where the authors derive the $\Gamma$-limit obtained in \cite{glp} starting from a nonlinear energy, under the assumption that $N_\e \approx |\log \e|$. It seems possible to adapt the techniques used in \cite{msz} to our case. This problem is currently under investigation by the authors.

Finally, a further step forward in our analysis is the following: in this paper the formation of polycrystalline structures is driven by relaxed boundary conditions, as usual for minimisation problems in $BV$ spacess. It would be interesting to deal with true boundary conditions, which  we expect to lead to the same $\Gamma$-limit $F^{g_A}$ defined in \eqref{gbd}. Moreover,  it would be interesting to replace boundary conditions by forcing terms. For instance, bulk forces in competition with  surface energies at grain boundaries should result in polycrystals exhibiting some intrinsic length scale. This is the case of semi-coherent interfaces, separated by periodic nets of dislocations (see \cite{fpp}).


\bibliography{bibliografia}

\bibliographystyle{plain}

\end{document}